\documentclass{amsart}
\usepackage{geometry}
\usepackage[utf8]{inputenc}  
\usepackage{amsthm}    
\usepackage{amsmath}   
\usepackage{bbm}    
\usepackage{amsfonts}   
\usepackage{amssymb}   
\usepackage{stmaryrd}   
\usepackage{comment}   
\usepackage{hyperref}   

\makeatletter
\DeclareRobustCommand\widecheck[1]{{\mathpalette\@widecheck{#1}}}
\def\@widecheck#1#2{%
    \setbox\z@\hbox{\m@th$#1#2$}%
    \setbox\tw@\hbox{\m@th$#1%
       \widehat{%
          \vrule\@width\z@\@height\ht\z@
          \vrule\@height\z@\@width\wd\z@}$}%
    \dp\tw@-\ht\z@
    \@tempdima\ht\z@ \advance\@tempdima2\ht\tw@ \divide\@tempdima\thr@@
    \setbox\tw@\hbox{%
       \raise\@tempdima\hbox{\scalebox{1}[-1]{\lower\@tempdima\box
\tw@}}}%
    {\ooalign{\box\tw@ \cr \box\z@}}}
\makeatother

\numberwithin{equation}{section}

\newtheorem{theorem}{Theorem}[section]

\newtheorem{lemma}[theorem]{Lemma}
\newtheorem{proposition}[theorem]{Proposition}

\theoremstyle{definition}

\theoremstyle{remark}
\newtheorem*{remark}{Remark}

\theoremstyle{remark}

\renewcommand{\Re}{\operatorname{Re}}
\renewcommand{\Im}{\operatorname{Im}}
\newcommand{\id}{\operatorname{id}}

\DeclareMathOperator{\GL}{GL}
\DeclareMathOperator{\SL}{SL}
\DeclareMathOperator{\PSL}{PSL}
\DeclareMathOperator{\SO}{SO}
\DeclareMathOperator{\Sp}{Sp}
\DeclareMathOperator{\tr}{tr}
\DeclareMathOperator{\Vol}{Vol}
\DeclareMathOperator{\Aut}{Aut}

\DeclareMathOperator{\Res}{Res}
\DeclareMathOperator{\sgn}{sgn}
\newcommand{\sumstar}{\sideset{}{^{*}}\sum}

\newcommand{\av}{\rm av}
\newcommand{\ev}{\rm ev}

\begin{document}
\begin{abstract}
We establish an asymptotic formula with a power-saving error of the $L^2$-norm of Siegel cusp forms of degree 2 in an average sense when restricted to the imaginary axis. The result is consistent with the Mass Equidistribution Conjecture for Siegel modular forms and the Lindelöf Hypothesis for some twisted Koecher-Maass series. Along the way, we perform a careful analysis of the Kitaoka formula of degree 2.
\end{abstract}


\title{A Restriction Norm Problem for Siegel Modular Forms}
\author{Gilles Felber}

\address{Max-Planck Institut for Mathematics, Vivatsgasse 7, 53111 Bonn, Germany}
\email{felber@mpim-bonn.mpg.de}

\thanks{Author supported by Germany's excellence strategy grant EXC-2047/1-390685813}

\maketitle

\section{Introduction}
Given a Riemannian manifold, a general question in analysis is to understand the size of the eigenfunctions of the associated Laplace-Beltrami operator. This can be investigated in different ways. For example, one can consider the distribution of mass of the function \cite{SoundQUE} or the relations between $L^p$-norms. It is also interesting to investigate the restriction of a function to some submanifold \cite{BGTRestNorm}. Generally, these questions are difficult and only stated as conjectures. A particularly interesting case is the one of an arithmetic manifold that is a symmetric space equipped with a family of Hecke operators. These commute with the Laplace-Beltrami operator and thus provide additional symmetries. It is then sometimes possible to have a better understanding of the situation than in the general case. It also links these questions to arithmetic problems.

In this paper, we investigate a similar setting for holomorphic Siegel modular forms of degree 2. These are the natural generalization of holomorphic modular forms for $\SL_2(\mathbb Z)$ to the symplectic group $\Sp_4(\mathbb Z)$. For a general introduction, see \cite{Klingen1990} or     \cite{Freitag1983}. Let $\mathbb H^{(2)}$ denote the space of $2\times 2$ symmetric complex matrices $Z=X+iY$ with the imaginary part positive definite, which we denote by $Y>0$. On this space, the symplectic group $\Sp_4(\mathbb Z)$ acts in a analogous way to $\SL_2(\mathbb Z)$. We consider the \emph{restriction $L^2$-norm} of a Siegel modular form to the imaginary axis $i\mathcal P(\mathbb R)$, where $\mathcal P(\mathbb R):=\{Y>0\}$. Namely, for a Hecke cusp form $f\in S_k^{(2)}$ of even weight $k$ we define
\begin{align}\label{defRestrictionNorm}
N(f):=\frac{\pi^2}{90\Vert f\Vert_2^2}\int_{\SL_2(\mathbb Z)\backslash\mathcal P(\mathbb R)}|f(iY)|^2\det(Y)^k\frac{dY}{\det(Y)^{3/2}}.
\end{align}

Here $\frac{dY}{\det(Y)^{3/2}}$ is the invariant measure over $\mathcal P(\mathbb R)$. In Equation (\ref{equIsoImaginaryAxisDetails}), we describe the measure-preserving isomorphism
\begin{align}\label{equIsoImaginaryAxis}
\mathcal P(\mathbb R)\cong\mathbb H\times\mathbb R_{>0}.
\end{align}

The group $\SL_2(\mathbb Z)$ acts on $\mathcal P(\mathbb R)$ by $U\cdot Y=U^tYU$ and this action is mapped by the isomorphism to the usual action on $\mathbb H$. We use the standard measures on the quotients $\SL_2(\mathbb Z)\backslash\mathbb H$ and $\Sp_4(\mathbb Z)\backslash\mathbb H^{(2)}$. In particular, these are not probability measures. The factor
$$\frac{\Vol(\Sp_4(\mathbb Z)\backslash\mathbb H^{(2)})}{\Vol(\SL_2(\mathbb Z)\backslash\mathbb H)}=\frac{\pi^3/270}{\pi/3}=\frac{\pi^2}{90}$$
takes this into account in the definition. We state the period formula for this norm. Let $\Lambda$ be a set of all the spectral components in the decomposition of $L^2(\SL_2(\mathbb Z)\backslash\mathbb H)$. It consists of the constant function, Eisenstein series and an orthonormal basis of Hecke-Maass cusp forms and is equipped with a measure $d\phi$ (see Section \ref{secNotations} for more details). We denote by $\Lambda_{\ev}$ the subset of $\Lambda$ consisting of even functions.

\begin{proposition}[\cite{Blomer2019symplectic}, Proposition 1]\label{proPeriodFormula}
Let $L(f\times\phi,s)$ be the Dirichlet series defined in Equation (\ref{defTwistedKMSeries}) and $G(f\times\phi,s)$ the corresponding gamma factor. Then
$$N(f)=\frac{\pi^2}{2880\Vert f\Vert_2^2}\int_{-\infty}^\infty\int_{\Lambda_{\ev}}|L(f\times\phi,1/2+it)G(f\times\phi,1/2+it)|^2d\phi\,dt.$$
\end{proposition}

We state now the main theorem of this article. Let $w:\mathbb R\to\mathbb R_{>0}$ be a smooth test function with support in $[1,2]$, $\omega=\int_1^2w(x)x^3dx$ and $B_k^{(2)}$ be a Hecke eigenbasis of the Siegel cusps forms of weight $k$ and degree 2. We define the following average over $B_k^{(2)}$ and $k$:
\begin{align}\label{defNav}
N_{\av}(K):=\frac{17280}{\omega K^4}\sum_{k\in2\mathbb N}w\left(\frac kK\right)\sum_{f\in B_k^{(2)}}N(f).
\end{align}

The constant in front is motivated by $|B_k^{(2)}|\sim\frac{k^3}{8640}$ and takes into account an extra factor $\frac12$ because we restrict ourselves to even $k$. We follow \cite{Kitaoka} in this restriction on $k$.

\begin{theorem}\label{thmNav}
Let $\epsilon>0$. We have
$$N_{\av}(K)=4\log(K)+C+O_\epsilon(K^{-1/2+\epsilon}),$$
for an explicit constant $C$ that only depends on $w$.
\end{theorem}

A similar problem was considered by Blomer and Corbett in \cite{Blomer2019symplectic}. There the average was done over a subspace of $S_k^{(2)}$ consisting of the Saito-Kurokawa lifts. These are lifts coming from half-integral weight modular forms in the Kohnen's plus space $S_{k-1/2}^+(4)$. The latter is in bijection with its Shimura lift to the classical modular forms $S_{2k-2}$. This allows the authors to reduce some computation to half-integral weight forms. The present text follows some ideas of that article and diverges from it at the application of a trace formula. In the latter, a relative trace formula for pairs of Heegner points is used, where here we use the Kitaoka formula, a generalization of the Petersson trace formula to higher degree Siegel modular forms. Nevertheless, some ideas remain valid in both cases and the careful reader can spot similarities along the whole article.

The result, and especially the constant 4, are interesting in two aspects. First, it fits in the more general question of Quantum Unique Ergodicity and its holomorphic counterpart, the Mass Equidistribution Conjecture \cite{HolowinskySound}. The latter is widely unknown for Siegel modular forms of degree bigger than 1. Still, one could hope it holds and even that a similar result is valid on a submanifold in the spirit of \cite{Young2016}. A heuristic argument is developed in \cite{Blomer2019symplectic} and shows the coherence of the constant 4 with the Mass Equidistribution Conjecture. In another direction, the result can be seen as an average version of the Lindelöf Hypothesis for the Dirichlet series of Proposition \ref{proPeriodFormula}. Although these series are not $L$-functions, as they lack an Euler product representation, one could still hope that the Lindelöf Hypothesis holds.

The proof relies on two trace formulas. The most important one is the Kitaoka formula for degree 2 forms, presented in Theorem \ref{thmKitaokaFormula} (see also \cite{Kitaoka}, \cite{Blomer2016spectral} Section 3). The general shape of the formula is similar to the Petersson formula, but the non-diagonal terms are given by a sum over integral matrices $C$. To simplify the analysis, the non-diagonal terms are split into matrices of ranks 1 and 2. Most difficulties arise for the rank 2 term, which features a generalized Bessel function of matrix argument (see also \cite{Herz1955}). The analysis of this involves, after averaging over $k$, an analysis of an oscillating integral and a careful count of matrices with close eigenvalues. The rank 1 term is simpler, because we understand the involved Bessel function. Nevertheless, we have to develop a non-trivial argument because a simple count of the number of terms with trivial bounds would not give us a power saving.

The main term of Theorem \ref{thmNav} is given by the diagonal term of the Kitaoka formula. For that part of the argument, we use the pre-trace formula to compute the spectral integral in Proposition \ref{proPeriodFormula}. We have to carefully count the terms on the geometric side of it to get a power saving. This corresponds to a bound on the number of Heegner points close to the boundary of a fundamental domain of $\SL_2(\mathbb Z)\backslash\mathbb H$. The average over $k$ is not used in this part and is only computed at the end. The dependence of $C$ on $w$ in Theorem \ref{thmNav} arises there and is explicit. We also point out that the restriction to Hecke eigenfunctions, both for the Siegel and the Maass forms does not play a role in the two formulas. It is used to get a self-dual approximate functional equation. We believe that the main results of this article hold without this condition.

The structure of the paper is the following. We begin by the proof of Proposition \ref{proPeriodFormula}. After that, we can average $N(f)$ and apply the Kitaoka formula. This is the content of Section \ref{section2}. After gathering some technical lemmas in Section \ref{section3}, we deal with the diagonal term and the pre-trace formula in Section \ref{section4}. Sections \ref{section5} and \ref{section6} consist of the analysis of respectively the rank 1 and rank 2 terms in the Kitaoka formula. In Appendix \ref{AppAutomorphisms}, we provide a table with all the $\GL_2(\mathbb Z)$-automorphisms of reduced positive definite quadratic forms and prove its correctness. This is used in the computation of the pre-trace formula for the even spectrum.

\subsection{Notations and normalizations}\label{secNotations}
We fix a few notations and normalize some objects in this section. What is written here is valid everywhere unless stated otherwise.

The set $\mathcal P(\mathbb R)$ consists of every 2 by 2 symmetric positive definite matrices and $\mathcal P(\mathbb Z)$ is the subset of elements with integral diagonal and half-integral non-diagonal elements. In particular $i\mathcal P(\mathbb R)$ corresponds to the imaginary axis of $\mathbb H^{(2)}$, i.e. the set of matrices with purely imaginary coordinates in $\mathbb H^{(2)}$. A matrix $\left(\begin{smallmatrix}a&b\\b&c\end{smallmatrix}\right)$ in $\mathcal P(\mathbb Z)$ (sometimes written $\left(\begin{smallmatrix}x&y\\y&z\end{smallmatrix}\right)$ or $(\begin{smallmatrix}\alpha&\beta\\\beta&\delta\end{smallmatrix})$) corresponds to a positive definite binary quadratic forms $ax^2+2bxy+dy^2$. Such a form is \emph{weakly reduced} if $2|b|\leq a\leq d$. It is \emph{reduced} if, moreover,
$$2|b|=a\text{ or }a=d\Longrightarrow b\geq0.$$
We always use the notation coming from matrices and do not introduce the alternative notations used for quadratic forms. In particular, we consider the determinant of the matrix but never the discriminant of the associated quadratic form.

A matrix $Y\in\mathcal P(\mathbb Z)$ also corresponds to a point $z_Y$ in $\mathbb H$, see Equation \ref{equIsoImaginaryAxisDetails}. The usual fundamental domain of $\SL_2(\mathbb Z)\backslash\mathbb H$ is $\{z=x+iy\in\mathbb C\mid-1/2\leq x\leq1/2,\ |z|\geq1\}$. A weakly reduced matrix $Y$ corresponds to a point in this domain. If $Y$ is reduced and the corresponding point is on the edge of the fundamental domain, then this point has non-positive real part.

Let $\ell=k-3/2$ and $f\in S_k^{(2)}$ be a Siegel cusp form of weight $k$ and degree 2. We only consider even weights. The Fourier series of $f$ is normalized in the following way:
$$f(Z)=\sum_{T\in\mathcal P(\mathbb Z)}a(T)\det(T)^{\ell/2}e(\tr(TZ)).$$

The set of spectral components is denoted by $\Lambda$ and $\Lambda_{\ev}$ is the subset of even forms. The eigenvalue of $\phi\in\Lambda$ is $\lambda_\phi$ and $t_\phi$ is the spectral parameter, given by $\lambda_\phi=\frac14+t_\phi ^2$. The constant function has spectral parameter $i/2$. We write $\int_\Lambda$ for the integral over the spectrum with the corresponding measure, that is $\frac{dt_\phi}{4\pi}$ for the continuous part and the counting measure for the discrete part. The symbol $\Gamma$ is only used for the gamma function and the modular group is designated by $\SL_2(\mathbb Z)$. We use the common notation $e(z)=e^{2\pi iz}$ and the Vinogradov symbols $\ll$, $\gg$, $\asymp$ and $\sim$.

Finally, we always assume that $k$ and $K$ are large enough and $\epsilon>0$ is small enough to avoid degenerate cases. We may change the value of $\epsilon$ from a display to the next, as long as the new $\epsilon$ is a constant multiple of the first one.

\subsection{Acknowledgments}
I would like to thank Valentin Blomer for his supervision of this project, his help to understand their article with A. Corbett and his constant support. I am also thankful to Edgar Assing, Farrell Brumley, Bart Michels, Raphael Steiner and Radu Toma for the enlightening discussions and advice.

\section{Restriction norm}\label{section2}
In this section, we prove Proposition \ref{proPeriodFormula} and introduce the Kitaoka formula. Most of the computations done here are well known and can be found in \cite{Blomer2019symplectic}, \cite{Blomer2016spectral} and other sources. We gather these results for convenience of the reader and provide detailed proofs of some of them. First, we describe in more detail the isomorphism (\ref{equIsoImaginaryAxis}). It is given by sending a matrix corresponding to a point in $\mathbb H$ to the matrix times its transpose. The other direction is given by sending a reduced positive quadratic form to the corresponding Heegner point. The last line is an explicit computation of the map. 
\begin{align}\label{equIsoImaginaryAxisDetails}
\mathbb H\times\mathbb R_{>0}\cong\SO(2)\backslash\SL_2(\mathbb R)\times\mathbb R_{>0}&\xrightarrow\sim\mathcal P(\mathbb R)\nonumber,\\
(\SO(2)\cdot g,r)&\mapsto r\cdot g^tg\nonumber,\\
\left(\frac{-b+i\sqrt{\det(M)}}a,\det(M)\right)&\mapsfrom M=\begin{pmatrix}a&b\\b&c\end{pmatrix}\nonumber,\\
(x+iy,r)&\mapsto\sqrt r\begin{pmatrix}y^{-1}&-xy^{-1}\\-xy^{-1}&y^{-1}(x^2+y^2)\end{pmatrix}.
\end{align}

This isomorphism preserves the measures $\left(\frac{dx\,dy}{y^2},\frac{dr}r\right)\mapsto\frac{dY}{\det(Y)^{3/2}}$ and the action of $\SL_2(\mathbb Z)$. Therefore we can take the quotient on both sides and write
$$N(f)=\frac{\pi^2}{90\Vert f\Vert_2^2}\int_{\SL_2(\mathbb Z)\backslash\mathbb H}\int_0^\infty|f(z,r)|^2\frac{dr}r\frac{dx\,dy}{y^2},$$
where $(z,r)\in\mathbb H\times\mathbb R_{>0}$ corresponds to a point in $\mathcal P(\mathbb R)$. If $Y$ is in $\mathcal P(\mathbb Z)$, the corresponding point $z_Y\in\mathbb H$ is called a \emph{Heegner point}.

\subsection{Spectral decomposition and Dirichlet series}
We now use the spectral decomposition of $L^2(\SL_2(\mathbb Z)\backslash\mathbb H)$. We denote by $\Lambda$ a set of spectral components. For any $g\in L^2(\SL_2(\mathbb Z)\backslash\mathbb H)$, we have the spectral decomposition $g(z)=\int_\Lambda\langle g,\phi\rangle \phi(z)d\phi$. We also have Parseval's identity
$$\int_{\SL_2(\mathbb Z)\backslash\mathbb H}|g(z)|^2\frac{dx\,dy}{y^2}=\int_\Lambda|\langle g,\phi\rangle|^2d\phi.$$

We wrote $f(z,r)$ as a function of $z\in\mathbb H$ and $r>0$. Let $\mathcal M(f)(z,r)$ denote the Mellin transform with respect to $r$. We apply Parseval's identity to the $z$ variable and the Mellin transform to the $r$ variable. This gives
\begin{align*}
\int_{\SL_2(\mathbb Z)\backslash\mathbb H}\int_0^\infty|f(z,r)|^2\frac{dr}r\frac{dx\,dy}{y^2}&=\int_\Lambda\int_0^\infty|\langle f(\cdot,r),\phi\rangle|^2\frac{dr}rd\phi\\
 &=\int_\Lambda\mathcal M\left(|\langle f(\cdot,0),\phi\rangle|^2\right)(0)d\phi\\
 &=\int_\Lambda\frac1{2\pi i}\int_{(0)}|\langle\mathcal M(f)(\cdot,s),\phi\rangle|^2ds\,d\phi.
\end{align*}

We used the formula $\mathcal M(|g|^2)(0)=\frac1{2\pi i}\int_{(0)}|\mathcal M(g)(s)|^2ds$ which is Parseval's theorem for the Mellin transform. Therefore the norm rewrites
$$N(f)=\frac{\pi^2}{90\Vert f\Vert_2^2}\frac1{2\pi i}\int_{(0)}\int_\Lambda|\langle\mathcal M(f)(\cdot,s),\phi\rangle|^2d\phi\,ds=\frac{\pi^2}{90\Vert f\Vert_2^2}\frac1{2\pi}\int_{-\infty}^\infty\int_\Lambda|\langle\mathcal M(f)(\cdot,1/2+it),\phi\rangle|^2d\phi\,dt.$$

The change of real part of the integral in the Mellin transform is valid and gives us nicer symmetry below. This Mellin transform can be computed explicitly. We define the \emph{twisted Koecher-Maass series}
\begin{align}\label{defTwistedKMSeries}
L(f\times\phi,s):=\sum_{T\in\mathcal P(\mathbb Z)/\PSL_2(\mathbb Z)}\frac{a(T)}{\epsilon(T)\det(T)^{1/4+s}}\phi(z_T),
\end{align}
where $a(T)$ is the $T$-th Fourier coefficient of $f$ and $\epsilon(T)=\#\{U\in\PSL_2(\mathbb Z)\mid U^tTU=T\}$ is the number of automorphisms of $T$. The corresponding gamma factor is
$$G(f\times\phi,s)=G(t_\phi,k,s):=4(2\pi)^{-(k-1)-2s}\Gamma\left(\frac\ell2+s+\frac{it_\phi}2\right)\Gamma\left(\frac\ell2+s-\frac{it_\phi}2\right).$$

Recall that $\ell=k-3/2$ and that $t_\phi$ is the spectral parameter of $\phi$. The series $L(f\times\phi,s)$ extends to an entire function that is bounded in vertical strips and we have the functional equation
$$\Lambda(f\times\phi,s):=L(f\times\phi,s)G(f\times\phi,s)=L(f\times\phi,1-s)G(f\times\phi,1-s).$$

However, it does not have an Euler product. The Mellin transform of $f$ is computed in \cite{Blomer2019symplectic}, Section 6. The result is
$$\langle\mathcal M(f)(\cdot,(k-1)/2+s),\phi\rangle=\int_0^\infty\langle f(\cdot,r),\phi\rangle r^{\frac{k-1}2+s}\frac{dr}r=\frac{\sqrt\pi}4L(f\times\bar\phi,s)G(f\times\bar\phi,s).$$

For odd $\phi$, the scalar product inside the $r$-integral vanishes. For even Hecke-Maass cusp forms $\phi$, we have $\bar\phi=\phi$ and for Eisenstein series, $|L(f\times\phi,s)|=|L(f\times\bar\phi,s)|$. Inserting this in the norm gives us
$$N(f)=\frac{\pi^2}{2880\Vert f\Vert_2^2}\int_{-\infty}^\infty\int_{\Lambda_{\ev}}|\Lambda(f\times\phi,1/2+it)|^2d\phi\,dt.$$

This concludes the proof of Proposition \ref{proPeriodFormula}.

\subsection{Approximate functional equation}
Now, we want to evaluate the series $L(f\times\phi,s)$ on the critical line using its Dirichlet series. For this, we compute an approximate functional equation. Note that if $f$ and $\phi$ are Hecke eigenfunctions, then $\overline{L(f\times\phi,s)}=L(f\times\phi,\bar s)$ for cusp forms and the constant function and
$$\overline{L(f\times E(\cdot,1/2+i\tau),s)}=L(f\times E(\cdot,1/2-i\tau),\bar s)=\nu(1/2-i\tau)L(f\times E(\cdot,1/2+i\tau),\bar s)$$
for Eisenstein series, where
$$\nu(s)=\pi^{1/2}\frac{\Gamma(s-1/2)}{\Gamma(s)}\frac{\zeta(2s-1)}{\zeta(2s)}=\frac{\pi^{-(1-s)}\Gamma(1-s)\zeta(2(1-s))}{\pi^{-s}\Gamma(s)\zeta(2s)}.$$

Let $\nu_\phi$ be $\nu(1/2-i\tau)$ with $\nu$ as above if $\phi$ an Eisenstein series and 1 if $\phi$ is a cusp form or the constant function. Then $\nu_\phi\phi(z)=\bar\phi(z)$ and $\overline{L(f\times\phi,s)}=\nu_\phi L(f\times\phi,\bar s)$. Consider
\begin{align*}
I(f\times\phi,s)&=\frac1{2\pi i}\int_{(3)}e^{v^2}\Lambda(f\times\phi,v+s)\Lambda(f\times\bar\phi,v+1-s)\frac{dv}v\\
	&=\frac1{2\pi i}\int_{(3)}e^{v^2}\nu_\phi\Lambda(f\times\phi,v+s)\Lambda(f\times\phi,v+1-s)\frac{dv}v.
\end{align*}

We take $s=1/2+it$. The integrand has no poles except for $v=0$ and decays rapidly at $\infty$. Moving the path of integration to $\Re(v)=-3$, we get
 $$I(f\times\phi,1/2+it)=\frac1{2\pi i}\int_{(-3)}e^{v^2}\nu_\phi\Lambda(f\times\phi,v+1/2+it)\Lambda(f\times\phi,v+1/2-it)\frac{dv}v+|\Lambda(f\times\phi,1/2+it)|^2.$$

Using the functional equation of $\Lambda(f\times\phi,s)$, we get
\begin{align*}
\frac1{2\pi i}\int_{(-3)}e^{v^2}&\nu_\phi\Lambda(f\times\phi,v+1/2+it)\Lambda(f\times\phi,v+1/2-it)\frac{dv}v\\
	&=\frac1{2\pi i}\int_{(-3)}e^{v^2}\nu_\phi\Lambda(f\times\phi,-v+1/2-it)\Lambda(f\times\phi,-v+1/2+it)\frac{dv}v\\
	&=-\frac1{2\pi i}\int_{(3)}e^{v^2}\nu_\phi\Lambda(f\times\phi,v+1/2-it)\Lambda(f\times\phi,v+1/2+it)\frac{dv}v\\
	&=-I(f\times\phi,1/2+it).
\end{align*}

We conclude that $|\Lambda(f\times\phi,1/2+it)|^2=2I(f\times\phi,1/2+it)$. Now, we expand the Dirichlet series of $L(f\times\phi,s)$ at $s=v+\frac12+it$:
\begin{align*}
I(f\times\phi,s)&=\frac1{2\pi i}\int_{(3)}e^{v^2}G(f\times\phi,v+1/2+it)G(f\times\phi,v+1/2-it)\\
 &\quad\cdot\sum_{T,Q\in\mathcal P(\mathbb Z)/\PSL_2(\mathbb Z)}\frac{a(T)a(Q)}{\epsilon(T)\epsilon(Q)\det(TQ)^{1/4+v+1/2}}\left(\frac{\det(Q)}{\det(T)}\right)^{it}\phi(z_T)\bar\phi(z_Q)\frac{dv}v\\
 &=\sum_{T,Q\in\mathcal P(\mathbb Z)/\PSL_2(\mathbb Z)}\frac{a(T)a(Q)}{\epsilon(T)\epsilon(Q)\det(TQ)^{3/4}}\phi(z_T)\bar\phi(z_Q)\\
 &\quad\cdot\left(\frac{\det(Q)}{\det(T)}\right)^{it}\frac1{2\pi i}\int_{(3)}e^{v^2}G(f\times\phi,v+1/2+it)G(f\times\phi,v+1/2-it)\det(TQ)^{-v}\frac{dv}v.
\end{align*}

This gives for the norm
\begin{align*}
N(f)&=\frac{\pi^2}{1440}\sum_{T,Q\in\mathcal P(\mathbb Z)/\PSL_2(\mathbb Z)}\frac1{\epsilon(T)\epsilon(Q)\det(TQ)^{3/4}}\int_{\Lambda_{\ev}}\int_{-\infty}^\infty\left(\frac{\det(Q)}{\det(T)}\right)^{it}\\
	&\quad\cdot\frac1{2\pi i}\int_{(3)}e^{v^2}G(t_\phi,k,v+1/2+it)G(t_\phi,k,v+1/2-it)(x_1x_2)^{-v}\frac{dv}v\,dt\,\phi(z_T)\bar\phi(z_Q)d\phi\\
	&\quad\cdot\frac{a(T)a(Q)}{\Vert f\Vert_2^2}.
\end{align*}
Note that only the last term depends on $f$ (for a fixed $k$).

\subsection{Average and Kitaoka formula}
We consider $N_{\av}$ as defined in Equation (\ref{defNav}). This is amenable to the Kitaoka formula. We begin by stating it. All notations are defined below. We take the description of all the functions from \cite{Blomer2016spectral}, Section 3. More details are given there for the interested reader.

\begin{theorem}[Kitaoka, \cite{Kitaoka}]\label{thmKitaokaFormula}
Let $k\geq6$, $B_k^{(2)}$ be a basis for the space of Siegel modular forms of degree 2 and even weight $k$. Then 
\begin{align*}
c_k\sum_{f\in\mathcal S_k^{(2)}}\frac{a_f(T)\overline{a_f(Q)}}{\Vert f\Vert_2^2}&=\delta_{Q\sim T}\#\Aut(T)\\
 &\quad+\sum_\pm\sum_{c,s\geq1}\sum_{U,V}\frac{(-1)^{k/2}\sqrt2\pi}{c^{3/2}s^{1/2}}H^\pm(UQU^t,V^{-1}TV^{-t};c)J_\ell\left(\frac{4\pi\sqrt{\det(TQ)}}{cs}\right)\\
 &\quad+8\pi^2\sum_{\det(C)\neq0}\frac{K(Q,T;C)}{|\det(C)|^{3/2}}\mathcal J_\ell(TC^{-1}QC^{-t}).
\end{align*}

\end{theorem}

\begin{remark}
This theorem generalizes Petersson formula to Siegel modular forms of degree 2. It is proved similarly by considering inner products of Poincaré series. On the right-hand side, the three terms are called, in order, the diagonal term, the rank 1 and the rank 2 terms. In these latter, we only use trivial bounds on the generalized Kloosterman sums $H^\pm$ and $K$. This is because we have short $s,c$ and $C$ sums up to some negligible errors.
\end{remark}

We define the symbols appearing in the theorem. First, we have the constants
$$c_k=2\sqrt\pi(4\pi)^{3-2k}\Gamma(k-3/2)\Gamma(k-2)$$
and $\ell=k-3/2$. We write $Q\sim T$ if $Q$ and $T$ are equivalent as quadratic forms, i.e. there exists $M\in\GL_2(\mathbb Z)$ such that $Q=M^tTM$. The set $\Aut(T)=\{M\in\GL_2(\mathbb Z)\mid T=M^t TM\}$ consists of all automorphisms of $T$ in $\GL_2(\mathbb Z)$.

In the rank one term, we sum over integers $c,s$ and matrices
\begin{align*}
U&\in\left\{\begin{pmatrix}*&*\\0&*\end{pmatrix}\right\}\backslash\GL_2(\mathbb Z),&
V&\in\GL_2(\mathbb Z)/\left\{\begin{pmatrix}1&*\\0&*\end{pmatrix}\right\}\\
\intertext{We define}
P&=UQU^t=\left(\begin{smallmatrix}p_1&p_2/2\\p_2/2&p_4\end{smallmatrix}\right),&
S&=V^{-1}TV^{-t}=\left(\begin{smallmatrix}s_1&s_2/2\\s_2/2&s_4\end{smallmatrix}\right).
\end{align*}

Suppose that the bottom right entries of $P$ and $S$ are both equal to $s$. In that case, we define
$$H^\pm(P,S;c):=\delta_{p_4=s_4}\sumstar_{d_1\bmod c}\sum_{d_2\bmod c}e\left(\frac{\bar d_1s_4d_2^2\mp\bar d_1p_2d_2+s_2d_2+\bar d_1p_1+d_1s_1}c\mp\frac{p_2s_2}{2cs_4}\right).$$
Here $\sumstar$ means that the sum is on $d_1$ coprime to $c$. We have the trivial bound $|H^\pm(P,S,c)|\leq c^2$. The function $J_\ell$ is the Bessel function of the first kind.

For the degree 2 term, $C\in M_2(\mathbb Z)$ runs over all matrices with non-zero determinant. We define
$$K(Q,T;C):=\sum e(\tr(AC^{-1}Q+C^{-1}DT)),$$
where the sum is over matrices $\left(\begin{smallmatrix}A&*\\C&D\end{smallmatrix}\right)$ in a system of representatives of $\Gamma_\infty\backslash\Sp_4(\mathbb Z)/\Gamma_\infty$ for a fixed $C$, where $\Gamma_\infty=\{\left(\begin{smallmatrix}1&X\\0&1\end{smallmatrix}\right)\in\Sp_4(\mathbb Z)\}$. A trivial bound is $|K(Q,T;C)|\leq|\det(C)|^{3/2}$. The function $\mathcal J_\ell$ is a generalized Bessel function defined in the following way. Let $P$ be a diagonizable matrix with positive eigenvalues $s_1^2,s_2^2$ ($s_1,s_2>0$). Then
$$\mathcal J_\ell(P):=\int_0^{\pi/2}J_\ell(4\pi s_1\sin(\theta))J_\ell(4\pi s_2\sin(\theta))\sin(\theta)d\theta.$$

\subsection{Cut-off}
We define
\begin{align}\label{defVfunction}
V(x_1,x_2,\tau,k):=\int_{-\infty}^\infty\left(\frac{x_2}{x_1}\right)^{it}\frac1{2\pi i}\int_{(3)}e^{v^2}c_k^{-1}G(\tau,k,v+1/2+it)G(\tau,k,v+1/2-it)(x_1x_2)^{-v}\frac{dv}vdt.
\end{align}

With this definition, we rewrite the norm in a compact way:
\begin{align}\label{equNormPreKitaoka}
N(f)&=\frac{\pi^2}{1440}\sum_{T,Q\in\mathcal P(\mathbb Z)/\PSL_2(\mathbb Z)}\frac1{\epsilon(T)\epsilon(Q)\det(TQ)^{3/4}}\nonumber\\
 &\quad\cdot\int_{\Lambda_{\ev}}V(\det(T),\det(Q),t_\phi,k)\phi(z_T)\bar\phi(z_Q)d\phi\frac{a(T)a(Q)}{\Vert f\Vert_2^2}.
\end{align}

\begin{lemma}
Let $x_1,x_2>0$, $\tau\in\mathbb C$ such that $|\Im(\tau)|\leq2$, $k$ big enough and $A>0$. The function $V$ satisfies the following bounds:
\begin{align}\label{eqcutoff}
\left(\frac{x_1}{\sqrt k}\right)^{j_1}\left(\frac{x_2}{\sqrt k}\right)^{j_2}k^{\frac12j_3+j_4}\frac{d^{j_1}}{dx_1^{j_1}}\frac{d^{j_2}}{dx_2^{j_2}}&\frac{d^{j_3}}{d\tau^{j_3}}\frac{d^{j_4}}{dk^{j_4}}V(x_1,x_2,\tau,k)\nonumber\\
 \ll_{A,j_1,j_2,j_3,j_4}&k^2\left(1+\frac{x_1x_2}{k^4}\right)^{-A}\left(1+k^{1/2}|\log(x_2/x_1)|\right)^{-A}\left(1+\frac{|\tau|^2}k\right)^{-A}.
\end{align}
\end{lemma}

\begin{remark}
This lemma is similar to Equations (10.5) in \cite{Blomer2019symplectic}. There is an error there in the derivatives of $x_1$ and $x_2$. The integral over $t$ add an extra $k^{1/2}$ term for each derivative. Note that the term is corrected when used later in Section 10.
\end{remark}

\begin{proof}
We can bound $G$ using the decay of the $\Gamma$ function. We establish the relevant bounds in the next section. First, we consider the inner integral
$$V_1(x,t,\tau,k)=\frac1{2\pi i}\int_{(3)}e^{v^2}c_k^{-1}G(\tau,k,v+1/2+it)G(\tau,k,v+1/2-it)x^{-v}\frac{dv}v$$
where $x>0$ and the other variables are as above. We prove
\begin{align}\label{equDecayV1}
x^{j_1}k^{\frac12j_2+\frac12j_3+j_4}\frac{d^{j_1}}{dx^{j_1}}\frac{d^{j_2}}{dt^{j_2}}\frac{d^{j_3}}{d\tau^{j_3}}\frac{d^{j_4}}{dk^{j_4}}V_1(x,t,\tau,k)\ll_{A,j_1,j_2,j_3,j_4}k^{3/2}\left(1+\frac x{k^4}\right)^{-A}\left(1+\frac{t^2+|\tau|^2}k\right)^{-A}.
\end{align}

This is similar to Equation (9.16) in \cite{Blomer2019symplectic}. All the derivatives except the one in $x$ are already treated in Lemma \ref{lemBoundGammaFactors}. First, we move the $v$-integral to a large real part $\Re(v)=A$. Then we apply Lemma \ref{lemBoundGammaFactors} and integrate by parts.
\begin{align*}
x^j\frac{d^j}{dx^j}&V_1(x,t,	\tau,k)\\
	&=\frac1{2\pi i}\int_{(A)}e^{v^2}c_k^{-1}G(\tau,k,v+1/2+it)G(\tau,k,v+1/2-it)x^{-v}(-v)\dots(-v-j+1)\frac{dv}v\\
	&=\frac1{2\pi i}\int_{(A)}e^{v^2}G_A(k,t,\tau,v)x^{-v}(-v)\dots(-v-j+1)\frac{dv}v+O_A(k^{-A})\\
	&\ll_Ak^{3/2+4A}\frac1{2\pi i}\int_{(A)}e^{-|v|^2}\left(1+\frac{t^2+|\tau|^2+\Im(v)^2}k\right)^{-A}x^{-A}|v|\dots|v+j-1|\frac{dv}{|v|}+O_A(k^{-A})\\
	&\ll_{A,j} k^{3/2}\left(\frac x{k^4}\right)^{-A}\left(1+\frac{t^2+|\tau|^2}k\right)^{-A}.
\end{align*}

Now, we move the $v$-integral to $\Re(v)=-\frac14$. If $j=0$, we pick up a pole at $v=0$.
\begin{align*}
x^j\frac{d^j}{dx^j}&V_1(x,t,\tau,k)\\
	&=\frac1{2\pi i}\int_{(-1/4)}e^{v^2}c_k^{-1}G(\tau,k,v+1/2+it)G(\tau,k,v+1/2-it)x^{-v}(-v)\dots(-v-j+1)\frac{dv}v\\
	&=\frac1{2\pi i}\int_{(-1/4)}e^{v^2}G_A(k,t,\tau,v)x^{-v}(-v)\dots(-v-j+1)\frac{dv}v+\delta_{j0}G_A(k,t,\tau,0)+O_A(k^{-A})\\
	&\ll_Ak^{1/2}\frac1{2\pi i}\int_{(-1/4)}e^{-|v|^2}\left(1+\frac{t^2+|\tau|^2+\Im(v)^2}k\right)^{-A}x^{1/4}|v|\dots|v+j-1|\frac{dv}{|v|}\\
	&\quad+k^{1/2}\left(1+\frac{t^2+|\tau|^2+\Im(v)^2}k\right)^{-A}+O_A(k^{-A})\\
	&\ll_A k^{3/2}\left(\left(\frac x{k^4}\right)^{1/4}+\delta_{j0}\right)\left(1+\frac{t^2+|\tau|^2}k\right)^{-A}.
\end{align*}

We conclude that Equation (\ref{equDecayV1}) holds. We consider now the $t$-integral. The only derivatives that we need to consider are the ones in $x_1$ and $x_2$. First, we integrate by parts.
\begin{align*}
V(x_1,x_2,\tau,k)&=\int_{-\infty}^\infty\left(\frac{x_2}{x_1}\right)^{it}V_1(x_1x_2,t,\tau,k)dt\\
	&=\int_{-\infty}^\infty(i\log(x_2/x_1))^{-j}\left(\frac{x_2}{x_1}\right)^{it}\frac{d^j}{dt^j}V_1(x_1x_2,t,\tau,k)dt\\
	&\ll_{A,j}k^{3/2}(k^{1/2}|\log(x_2/x_1)|)^{-j}\left(1+\frac{x_1x_2}{k^4}\right)^{-A}\int_{-\infty}^\infty\left(1+\frac{t^2+|\tau|^2}k\right)^{-A}dt\\
	&\ll_{A,j}k^2(k^{1/2}|\log(x_2/x_1)|)^{-j}\left(1+\frac{x_1x_2}{k^4}\right)^{-A}\left(1+\frac{|\tau|^2}k\right)^{-A}.
\end{align*}

Considering $j=0$ and $j=A$, we get the correct result for $j_1=j_2=0$ in Equation (\ref{eqcutoff}). If we differentiate with respect to $x_1$ or $x_2$, we get an extra factor $\frac1{x_1}$ resp. $\frac1{x_2}$ and another factor of size either 1 or $k^{1/2}$. We conclude that the result holds.
\end{proof}

\section{Technical lemmas}\label{section3}
We gather here various estimates and lemmas for the rest of the article. Most of them come from Section 6 of \cite{Blomer2019symplectic}.

\subsection{Gamma factors}
\begin{lemma}[\cite{Blomer2019symplectic}, Lemma 22]\label{lemGammaFunctionBounds}
Let $k\geq1$, $s=\sigma+it$ such that $k+\sigma\geq1/2$ and $A\in\mathbb N$, $i,j\in\mathbb N_0$. Then
$$\frac{\Gamma(k+s)}{\Gamma(k)}=k^sG_{A,\sigma}(k,t)+O_{A,\sigma}((k+|t|)^{-A}),$$
where
$$k^{j_1+j_2/2}\frac{d^{j_1}}{dk^{j_1}}\frac{d^{j_2}}{dt^{j_2}}G_{A,\sigma}(k,t)\ll_{A,\sigma,j_1,j_2}\left(1+\frac{t^2}k\right)^{-A}.$$
Moreover,
$$\frac{\Gamma(k+s)}{\Gamma(k)}=k^s\exp\left(-\frac{t^2}{2k}\right)\left(1+O_\sigma\left(\frac{|t|}k+\frac{t^4}{k^3}\right)\right).$$
\end{lemma}

\begin{lemma}[\cite{Blomer2019symplectic}, similar to Corollary 23]\label{lemBoundGammaFactors}
Let $A\geq0$, $\sigma\geq-1/4$, $t\in\mathbb R$, $\tau\in\mathbb C$ such that $|\Im(\tau)|\leq2$ and $k\in2\mathbb N$. Then
$$c_k^{-1/2}G(\tau,k,\sigma+1/2+it)\ll_{A,\sigma}k^{2\sigma+3/4}\left(1+\frac{t^2+|\tau|^2}k\right)^{-A}$$
We also have, for $v\in\mathbb C$,
$$c_k^{-1}G(\tau,k,v+1/2+it)G(\tau,k,v+1/2-it)=G_A(k,t,\tau,v)+O
_{\Re(v),A}(k^{-A}),$$
where
$$k^{j_1+j_2/2+j_3/2}\frac{d^{j_1}}{dk^{j_1}}\frac{d^{j_2}}{dt^{j_2}}\frac{d^{j_3}}{d\tau^{j_3}}G_A(k,t,\tau,v)\ll_{A,j_1,j_2,j_3,\Re(v)}k^{3/2+4\Re(v)}\left(1+\frac{t^2+|\tau|^2+\Im(v)^2}k\right)^{-A}.$$
Moreover, for $t,\tau\ll k^{1/2+\epsilon}$, we have
$$c_k^{-1}G(\tau,k,1/2+it)G(\tau,k,1/2-it)=\frac2{\pi^{5/2}}k^{3/2}\exp\left(-\frac{4t^2+\tau^2}k\right)\left(1+O(k^{-1/2+\epsilon})\right).$$
\end{lemma}

\begin{proof}
Using the Legendre duplication formula, $\Gamma(z)\Gamma(z+1/2)=2^{1-2z}\sqrt\pi\Gamma(2z)$, we get
\begin{align*}
c_k&=2\sqrt\pi(4\pi)^{3-2k}\Gamma(k-3/2)\Gamma(k-2)\\
 &=2\sqrt\pi(4\pi)^{3-2k}2^{2k-7/2-2}\pi^{-1}\Gamma\left(\frac{k-3/2}2\right)\Gamma\left(\frac{k-1/2}2\right)\Gamma\left(\frac{k-2}2\right)\Gamma\left(\frac{k-1}2\right)\\
 &=2^{-1}(2\pi)^{5/2-2k}\Gamma\left(\frac{k-3/2}2\right)\Gamma\left(\frac{k-1/2}2\right)\Gamma\left(\frac{k-2}2\right)\Gamma\left(\frac{k-1}2\right).
\end{align*}

Taking the square of the gamma factor, this gives
\begin{align*}
c_k^{-1}G(\tau,k,\sigma+1/2+it)^2&=\frac{2^4(2\pi)^{-2k-4\sigma-4it}}{2^{-1}(2\pi)^{5/2-2k}}\frac{\Gamma\left(\frac{k-1/2}2+\sigma+it+\frac{i\tau}2\right)^2\Gamma\left(\frac{k-1/2}2+\sigma+it-\frac{i\tau}2\right)^2}{\Gamma\left(\frac{k-3/2}2\right)\Gamma\left(\frac{k-1/2}2\right)\Gamma\left(\frac{k-2}2\right)\Gamma\left(\frac{k-1}2\right)}\\
 &=2^5(2\pi)^{-4(\sigma+it)-5/2}\frac{\Gamma\left(\frac{k-1/2}2+\sigma+it+\frac{i\tau}2\right)^2\Gamma\left(\frac{k-1/2}2+\sigma+it-\frac{i\tau}2\right)^2}{\Gamma\left(\frac{k-3/2}2\right)\Gamma\left(\frac{k-1/2}2\right)\Gamma\left(\frac{k-2}2\right)\Gamma\left(\frac{k-1}2\right)}.\\
\intertext{Applying Lemma \ref{lemGammaFunctionBounds}, we get}
 &\ll_{A,\sigma}(k/2)^{4\sigma+1/2+0+3/4+1/4}\left(1+\frac{(t+\tau)^2}k\right)^{-A}\\
 &\ll_{A,\sigma}k^{4\sigma+3/2}\left(1+\frac{t^2+|\tau|^2}k\right)^{-A}.
\end{align*}

This gives the first formula. Similarly, for the second formula, we compute
\begin{align*}
c_k^{-1}G(\tau,k&,v+1/2+it)G(\tau,k,v+1/2-it)\\
	&=2^5(2\pi)^{-4v-5/2}\frac{\Gamma\left(\frac{k-1/2}2+v+it+\frac{i\tau}2\right)\Gamma\left(\frac{k-1/2}2+v+it-\frac{i\tau}2\right)}{\Gamma\left(\frac{k-3/2}2\right)\Gamma\left(\frac{k-1/2}2\right)}\\
	&\quad\cdot\frac{\Gamma\left(\frac{k-1/2}2+v-it+\frac{i\tau}2\right)\Gamma\left(\frac{k-1/2}2+v-it-\frac{i\tau}2\right)}{\Gamma\left(\frac{k-2}2\right)\Gamma\left(\frac{k-1}2\right)}\\
	&=2^5(2\pi)^{-4v-5/2}(k/2)^{4v+3/2}G_{A,\sigma}(k,t,\tau,\Im(v))+O_{A,\sigma}\left((k+|t|+|\tau|)^{-A}\right),
\end{align*}
where $G_{A,\sigma}$ is the combination of the functions $G_{A,\sigma}$ in Lemma \ref{lemGammaFunctionBounds} for the four ratios of gamma functions. We have the following properties for the $G_{A,\sigma}$ function:
\begin{align*}
k^{j_1+j_2/2+j_3/2}\frac{d^{j_1}}{dk^{j_1}}\frac{d^{j_2}}{dt^{j_2}}\frac{d^{j_3}}{d\tau^{j_3}}G_{A,\sigma}(k,t,\tau,\Im(v))\ll_{A,\sigma,j_1,j_2,j_3}\left(1+\frac{\Im(v)^2+t^2+|\tau|^2}k\right)^{-A}.
\end{align*}

We used that if $k\ll\Im(v)^2+t^2+|\tau|^2$, then $k\ll(\Im(v)\pm t\pm\tau)^2$ for one of the choices of signs. The last equation comes from the corresponding formula in Lemma \ref{lemGammaFunctionBounds}. We get
\begin{align*}
c_k^{-1}&G(\tau,k,1/2+it)G(\tau,k,1/2-it)\\
	&=2^5(2\pi)^{-5/2}\frac{\Gamma\left(\frac{k-1/2}2+it+\frac{i\tau}2\right)\Gamma\left(\frac{k-1/2}2+it-\frac{i\tau}2\right)\Gamma\left(\frac{k-1/2}2-it+\frac{i\tau}2\right)\Gamma\left(\frac{k-1/2}2-it-\frac{i\tau}2\right)}{\Gamma\left(\frac{k-3/2}2\right)\Gamma\left(\frac{k-1/2}2\right)\Gamma\left(\frac{k-2}2\right)\Gamma\left(\frac{k-1}2\right)}\\
	&=2^5(2\pi)^{-5/2}(k/2)^{3/2}\exp\left(-\frac{2(t+\tau/2)^2+2(t-\tau/2)^2}k\right)\left(1+O(k^{-1/2+\epsilon})\right)\\
	&=\frac{2k^{3/2}}{\pi^{5/2}}\exp\left(-\frac{4t^2+\tau^2}k\right)\left(1+O(k^{-1/2+\epsilon})\right).
\end{align*}

\end{proof}

\subsection{The $J$ Bessel function and the spectral integral}
Concerning the $J$ Bessel function, we need the estimates
\begin{align}\label{boundsJBessel}
J_k(x)\ll&1, &J_k(x)\ll&\left(\frac xk\right)^k, &J_k(x)\ll&x^{-1/2}.
\end{align}
The first two are valid for $x>0$ and $k>2$ and the last one for $x\geq2k$ as stated in Equations (4.1), (4.2) and (4.3) of \cite{Blomer2016spectral}. Moreover Equation (4.7) in the same article says that the product of two Bessel functions can be rewritten in the following way:
\begin{align}\label{equProductOfBesselFunctions}
J_k(4\pi s_1\sin(\alpha))J_k(4\pi s_2\sin(\alpha))&=\frac1\pi\Re\left(e\left(-\frac{k+1}4\right)\int_0^\infty e\left((s_1^2+s_2^2)t+\frac{\sin(\alpha)^2}t\right) J_k(4\pi s_1s_2t)\frac{dt}t\right).
\end{align}

The following lemma is used to take advantage of the average over $k$.
\begin{lemma}[\cite{Blomer2019symplectic}, Lemma 20 and the remark after]\label{lemSumK}
Let $x>0$, $A\geq0$, $K>1$, $w:\mathbb R\to\mathbb C$ smooth with support in $[1,2]$, such that $w^{(j)}(x)\ll_\epsilon K^{j\epsilon}$. Then there exist smooth functions $w_0,w_-$ and $w_+$ such that for all $j\in\mathbb N_0$, we have

$$\sum_{k\text{ even}}i^kw\left(\frac kK\right)J_{k-3/2}(x)=w_0(x)+e^{ix}w_+(x)+e^{-ix}w_-(x)$$

and
\begin{align*}
 w_0(x)&\ll_AK^{-A},\\
 \frac{d^j}{dx^j}w_\pm(x)&\ll_{j,A}\left(1+\frac{K^2}x\right)^{-A}\frac1{x^j}
\end{align*}

Moreover, if $w$ depends on other parameters with control over the derivatives, so do $w_\pm$. We also have $w_0(x),w_\pm(x)\ll x^{-1/2}$.
\end{lemma}

We state a general upper bound for the spectral integral. 
\begin{lemma}\label{lemBoundNonDiagSpect}
Let $z_1,z_2\in\mathbb H$ with $\Im(z_1),\Im(z_2)\gg T$ and $T\geq1$. Then
\begin{align*}
\int_{\substack{\Lambda_{\ev}\\|t_\phi|\ll T}}|\phi(z_1)\phi(z_2)|d\phi\ll_A T\sqrt{\Im(z_1)\Im(z_2)}.
\end{align*}
\end{lemma}

\begin{proof}
We apply the Cauchy-Schwarz inequality. We get
\begin{align*}
\int_{\substack{\Lambda_{\ev}\\|t_\phi|\ll T}}|\phi(z_1)\phi(z_2)|d\phi&\ll\int_{\substack{\Lambda\\|t_\phi|\ll T}}|\phi(z_1)\phi(z_2)|d\phi\\
	&\ll\left(\int_{\substack{\Lambda\\|t_\phi|\ll T}}|\phi(z_1)|^2d\phi\right)^{1/2}\left(\int_{\substack{\Lambda\\|t_\phi|\ll T}}|\phi(z_2)|^2d\phi\right)^{1/2}.\\
\intertext{We bound the two terms with Proposition 15.8 in \cite{IwaniecKowalski}. The hypothesis gives}
	&\ll\left(T^2+T\Im(z_1)\right)^{1/2}\left(T^2+T\Im(z_2)\right)^{1/2}\\
	&\ll T\sqrt{\Im(z_1)\Im(z_2)}.
\end{align*}
\end{proof}

\subsection{Stationary phase}
We state in this section two lemma from \cite{BKY2013} about estimates on oscillating integral. Let $w$ be a smooth function with support on $[\alpha,\beta]$ and $h$ be a smooth functions on $[\alpha,\beta]$. We want to bound the integral
$$I=\int_{-\infty}^\infty w(t)e^{ih(t)}dt.$$
This depends on the vanishing of $h'$ in the interval $[\alpha,\beta]$.

\begin{lemma}[\cite{BKY2013}, Lemma 8.1]\label{lemBKY8.1}
Let $Y\geq1$, $X,U,R,Q>0$. Suppose that
\begin{align*}
w^{(j)}(t)&\ll_jXU^{-j},&\text{for }j=1,2,\dots\\
|h'(t)|&\geq R,\\
h^{(j)}(t)&\ll_jYQ^{-j},&\text{for }j=2,3,\dots
\end{align*}
Then
$$I\ll_A(\beta-\alpha)X[(QR/\sqrt y)^{-A}+(RU)^{-A}].$$
\end{lemma}

\begin{lemma}[\cite{BKY2013}, Proposition 8.2]\label{lemBKY8.2}
Let $0<\delta<1/10$, $X,U,Y,Q>0$, $Z=Q+X+Y+\beta-\alpha+1$ be such that
$$y\geq Z`{3\delta},\quad\beta-\alpha\geq U\geq\frac{QZ^{\delta/2}}{\sqrt Y}.$$
Suppose that
\begin{align*}
w^{(j)}(t)&\ll_j XU^{-j},&\text{for }j=0,1,\dots\\
h''(t)&\gg YQ^{-2},\\
h^{(j)}(t)&\ll_jYQ^{-j},&\text{for }j=1,2,\dots
\end{align*}
and that there exists a unique $t_0\in[\alpha,\beta]$ such that $h'(t_0)=0$. Then
$$I\ll\frac{QX}{\sqrt Y}.$$
\end{lemma}

\section{Diagonal term}\label{section4}
In this section, we compute the diagonal term of the Kitaoka formula, where $T\sim Q$. We have in particular that $\det(T)=\det(Q)$ and $\epsilon(T)=\epsilon(Q)$. For this, we combine Equations (\ref{defNav}) and (\ref{equNormPreKitaoka}) with the Kitaoka formula. The equation for the diagonal term simplifies to
\begin{align*}
N_{\av}^{\rm diag}(K)=&\frac{12\pi^2}{\omega K^4}\sum_{k\in2\mathbb N}w\left(\frac kK\right)\sum_{T\in\mathcal P(\mathbb Z)/\PSL_2(\mathbb Z)}\frac{\#\Aut(T)}{\epsilon(T)^2\det(T)^{3/2}}\\
 &\cdot\int_{\Lambda_{\ev}}V(\det(T),\det(T),t_\phi,k)\left|\phi(z_T)\right|^2d\phi.
\end{align*}

We deal with this expression in the following steps. First, we analyze the spectral integral using the pre-trace formula. This requires a non-trivial argument and takes up most of this section. In particular, we need to count Heegner points close to the edge of the fundamental domain and take care of the restriction to the even spectrum. After that, the rest of the summations and estimations is handled. We finish by the computation of the average over $k$, which we do not take advantage of for this term.

We fix some notations for this section. Let $T$ be a reduced positive definite matrix of determinant $D$. It corresponds to a Heegner $z_T$ in the fundamental domain via Equation (\ref{equIsoImaginaryAxisDetails}). We use the notation
$$T=\left(\begin{matrix}\alpha&\beta\\ \beta&\delta\end{matrix}\right)\longleftrightarrow z_T=\frac{-\beta+i\sqrt D}\alpha.$$

\subsection{The pre-trace formula}
We apply the pre-trace formula (see for example \cite{IwaniecSpectralMethods}, Section 10.1). At first, we do not take the even spectrum into account. It gives
\begin{align*}
\int_\Lambda V(\det(T),\det(T),t_\phi,k)\left|\phi(z_T)\right|^2d\phi&=\sum_{\gamma\in\SL_2(\mathbb Z)}\kappa(u(z_T,\gamma z_T)).
\end{align*}

The function $\kappa$ is the Harish-Chandra inverse of $V$ and it only depends on the point pair invariant $u(z_1,z_2)=\frac{|z_1-z_2|^2}{4\Im(z_1)\Im(z_2)}$. We keep these notations in this chapter. Recall that $T$ is reduced, so $2|\beta|\leq\alpha\leq\delta$ and the decay properties of $V$ give $D\ll k^{2+\epsilon}$, up to a negligible error. In particular, $z$ is in the classical fundamental domain of $\Gamma\backslash\mathbb H$. For the edge of the domain, we pick the pieces with $\Re(z)\leq0$. The goal of the following subsections is to prove the following theorem

\begin{theorem}\label{thmPretraceEstimate}
For all $\epsilon>0$, we have
\begin{align}
\sum_{T\in\mathcal P(\mathbb Z)/\PSL_2(\mathbb Z)}\frac{\#\Aut(T)}{\epsilon(T)^2\det(T)^{3/2}}&\int_\Lambda V(\det(T),\det(T),\tau,k)\left|\phi(z_T)\right|^2d\phi\nonumber\\
 =&\sum_{T\in\mathcal P(\mathbb Z)/\PSL_2(\mathbb Z)}\frac{2\#\Aut(T)}{\epsilon(T)\det(T)^{3/2}}\kappa(0)+O_\epsilon(k^{2.5+\epsilon}),\label{equAsymptoteFullSpectrum}
\intertext{If we reduce to the even spectrum, we get}
\sum_{T\in\mathcal P(\mathbb Z)/\PSL_2(\mathbb Z)}\frac{\#\Aut(T)}{\epsilon(T)^2\det(T)^{3/2}}&\int_{\Lambda_{\ev}} V(\det(T),\det(T),\tau,k)\left|\phi(z_T)\right|^2d\phi\nonumber\\
 =&\sum_{T\in\mathcal P(\mathbb Z)/\PSL_2(\mathbb Z)}\left(\frac{\#\Aut(T)}{\epsilon(T)}\right)^2\frac1{2\det(T)^{3/2}}\kappa(0)+O_\epsilon(k^{2.5+\epsilon}),\nonumber
\end{align}
\end{theorem}

\begin{remark}
We will see that the term with $\kappa(0)$ is of size $k^3\log(k)$. The two equations tells us that, on average over $T$, the terms on the spectral side with $u\neq0$ are of lower size.
\end{remark}

To get $\kappa$, we want to compute the Harish-Chandra inverse transform of the function $h(\tau) = V(\det(T), \det(T), \tau, k)$. We know that this function of $\tau$ is even, decays exponentially and is holomorphic in the strip $|\Im(\tau)|\leq 2$. Therefore it is suitable for the Harish-Chandra inversion. A first way to get $\kappa$ in terms of $h$ is given by (1.62') in \cite{IwaniecSpectralMethods}:
$$\kappa(u)=\frac1{4\pi}\int_{-\infty}^\infty\frac1\pi\int_0^\pi(2u+1+2\sqrt{u(u+1)}\cos(\theta))^{-\frac12-i\tau}d\theta\  h(\tau)\tau\tanh(\pi\tau)d\tau.$$

At $u=0$, we get
\begin{align}\label{equKappa0}
\kappa(0)&=\frac1{4\pi}\int_{-\infty}^\infty V(\det(T),\det(T),\tau,k)\tau\tanh(\pi\tau)d\tau.
\end{align}

For $u\ll1$, the $\theta$-integral is of size $\ll1$ because $2u+1-2\sqrt{u(u+1)}$ is bounded away from 0. Using the cut-off of $V$ given in Equation (\ref{eqcutoff}) and $\tau\tanh(\tau)=|\tau|+O(1)$, we get a trivial bound
\begin{align}\label{equBoundKappaTrivial}
\kappa(u)\ll_A k^2\left(1+\frac{\det(T)^2}{k^4}\right)^{-A}\int_{-\infty}^\infty\left(1+\frac{|\tau|^2}k\right)^{-A}d\tau\ll_A k^3\left(1+\frac{\det(T)^2}{k^4}\right)^{-A}.
\end{align}

We also need the following lemma.
\begin{lemma}[\cite{IwaniecSpectralMethods}, Lemma 2.11]\label{lemIwaniecNumberOfgammas}
Let $z\in\mathbb H$ with $\Im(z)\geq1/10$ and $X>0$. We have
\begin{align*}
\#\{\gamma\in\SL_2(\mathbb Z)\mid u(z,\gamma z)<X\}\ll&\sqrt{X(X+1)}\Im(z)+X+1,\\
\#\{\gamma\in\SL_2(\mathbb Z)\mid u(z,\gamma(-\bar z))<X\}\ll&\sqrt{X(X+1)}\Im(z)+X+1.
\end{align*}
\end{lemma}

\begin{remark}
Note that in our case, $\Im(z_T)=\frac{\sqrt D}\alpha\ll k^{1+\epsilon}$ up to a negligible error.
\end{remark}

\subsection{Decay of the Harish-Chandra inverse transform}\label{secDecayHCTransform}
We do not restrict to the even spectrum at first. In this section, we prove a strong decay bound for $\kappa(u)$ when $u$ is big enough.
\begin{lemma}\label{lemDecayOfKappa}
Let $\epsilon>0$, $A>0$, $T\in\mathcal P(\mathbb Z)$ with $\det(T)\ll k^{2+\epsilon}$ and $z_T\in\mathbb H$ the Heegner point corresponding to $T$ via Equation (\ref{equIsoImaginaryAxisDetails}). Then
$$\sum_{\substack{\gamma\in\SL_2(\mathbb Z)\\u(z_T,\gamma z_T)\geq k^{-1+\epsilon}}}|\kappa(u(z_T,\gamma z_T))|\ll_{A,\epsilon} k^{-A}.$$
\end{lemma}

\begin{proof}
We apply the usual three steps to get the Harish-Chandra inverse transform (see (1.64) in \cite{IwaniecSpectralMethods}). This gives
\begin{align*}
g(r)&=\frac1{2\pi}\int_{-\infty}^\infty e^{ir\tau}V(\det(T),\det(T),\tau,k)d\tau,\\
q(v)&=\frac1{4\pi}\int_{-\infty}^\infty V(\det(T),\det(T),\tau,k)(\sqrt{v+1}+\sqrt v)^{2i\tau}d\tau,\\
\kappa(u)&=\frac1{4\pi^2i}\int_u^\infty\frac1{\sqrt{v-u}}\int_{-\infty}^\infty V(\det(T),\det(T),\tau,k)\frac{(\sqrt{v+1}+\sqrt v)^{2i\tau}}{\sqrt{v(v+1)}}\tau d\tau\,dv.
\end{align*}

We recall the decay property of $V$ with respect to $\tau$, as written in Equation (\ref{eqcutoff}):
$$\frac{d^j}{d\tau^j}V(\det(T),\det(T),\tau,k)\ll_{A,j}k^{2-j/2}\left(1+\frac{\det(T)^2}{k^4}\right)^{-A}\left(1+\frac{|\tau|^2}k\right)^{-A}.$$

Let $h(\tau)=V(\det(T),\det(T),\tau,k)$. We consider first $q(v)$.
Since $h$ is holomorphic in a strip, we can move the integration line to $\tau\mapsto\tau+2i$:
$$\int_{-\infty}^\infty h(\tau)\tau(\sqrt{v+1}+\sqrt{v})^{2i\tau}d\tau=(\sqrt{v+1}+\sqrt{v})^{-4}\int_{-\infty}^\infty h(\tau+2i)(\tau+2i)(\sqrt{v+1}+\sqrt{v})^{2i\tau}d\tau.$$

Integrating by parts, we get
\begin{align*}
4\pi q(v)&=(\sqrt{v+1}+\sqrt{v})^{-4}\int_{-\infty}^\infty h(\tau+2i)(\tau+2i)(\sqrt{v+1}+\sqrt{v})^{2i\tau}d\tau\\
	&=(\sqrt{v+1}+\sqrt{v})^{-4}(-2i\log(\sqrt{v+1}+\sqrt{v}))^{-1}\\
	&\quad\cdot\int_{-\infty}^\infty(h'(\tau+2i)(\tau+2i)+h(\tau+2i))(\sqrt{v+1}+\sqrt{v})^{2i\tau}d\tau\\
	&=(\sqrt{v+1}+\sqrt{v})^{-4}(-2i\log(\sqrt{v+1}+\sqrt{v}))^{-j}\\
	&\quad\cdot\int_{-\infty}^\infty (h^{(j)}(\tau+2i)(\tau+2i)+jh^{(j-1)}(\tau+2i))(\sqrt{v+1}+\sqrt{v})^{2i\tau}d\tau\\
	&\ll_{A,j}(\sqrt{v+1}+\sqrt v)^{-4}\left(\log(\sqrt{v+1}+\sqrt v)\sqrt k\right)^{-j}k^2(k+j\sqrt k)\left(1+\frac{\det(T)^2}{k^4}\right)^{-A}.
\end{align*}

In particular, we have a saving in $k$ if $\log(\sqrt{v+1}+\sqrt v)\gg k^{-1/2+\epsilon/2}$. Since $\log(\sqrt{v+1}+\sqrt v)=\sqrt v+O(v^{3/2})$ for small $v$, this happens if $v$ or $u$ is $\gg k^{-1+\epsilon}$. We obtain
$$q(v)\ll_{A,j}(\sqrt{v+1}+\sqrt v)^{-4}k^{3-j\epsilon/2}\left(1+\frac{\det(T)^2}{k^4}\right)^{-A}.$$

Then
$$\kappa(u)\ll_{A,j}k^{3-j\epsilon/2}\left(1+\frac{\det(T)^2}{k^4}\right)^{-A}\int_u^\infty\frac{dv}{\sqrt{v(v+1)(v-u)}(\sqrt{v+1}+\sqrt v)^4}.$$

We split the integral in the intervals $]u,u+1[$ and $[u+1,\infty[$. We get
\begin{align*}
\int_u^\infty\frac{dv}{\sqrt{v(v+1)(v-u)}(\sqrt{v+1}+\sqrt v)^4}\ll&\frac1{\sqrt{u(u+1)}(\sqrt{u+1}+\sqrt u)^4}\int_u^{u+1}\frac{dv}{\sqrt{v-u}}+\int_{u+1}^\infty\frac{dv}{v^3}\\
 =&\frac2{\sqrt{u(u+1)}(\sqrt{u+1}+\sqrt u)^4}+\frac1{2(u+1)^2}.
\end{align*}

If $u\gg1$, then we obtain
$$\frac2{\sqrt{u(u+1)}(\sqrt{u+1}+\sqrt u)^4}+\frac1{2(u+1)^2}\ll\frac1{u^2}.$$

If $u\ll1$, then we have $1+u\asymp 1$ and
$$\frac2{\sqrt{u(u+1)}(\sqrt{u+1}+\sqrt u)^4}+\frac1{2(u+1)^2}\ll\frac1{\sqrt u}+1.$$

In summary, for $u\gg k^{-1+\epsilon}$, we computed
\begin{align*}
\kappa(u)\ll&_{A,j}k^{3-j\epsilon/2}\left(1+\frac{\det(T)^2}{k^4}\right)^{-A}\frac1{u^2} &\text{if }u\gg1,\\
\kappa(u)\ll&_{A,j}k^{3-j\epsilon/2}\left(1+\frac{\det(T)^2}{k^4}\right)^{-A}	&\text{if }u\ll1.
\end{align*}

Applying Lemma \ref{lemIwaniecNumberOfgammas}, we sum over $\gamma$. For $k^{-1+\epsilon}\ll u_T\ll1$, we have
$$\sum_{\substack{\gamma\in\SL_2(\mathbb Z)\\k^{-1+\epsilon}\ll u(z_T,\gamma z_T)\ll1}}|\kappa(u(z_T,\gamma z_T))|\ll_{A,j}k^{4.5-j\epsilon/2}\left(1+\frac{\det(T)^2}{k^4}\right)^{-A}$$

So for $j$ big enough, we can cancel all the powers of $k$. For $u\gg1$, we split into dyadic intervals. We can begin the sum at say $1$. For $X\gg1$, we have $\sqrt{X(X+1)}\Im(z)+X+1\ll Xk^{1+\epsilon}$. We get
\begin{align*}
\sum_{\substack{\gamma\in\SL_2(\mathbb Z)\\u(z_T,\gamma z_T)\geq1}}|\kappa(u(z_T,\gamma z_T))|=&\sum_{n=0}^\infty\sum_{u\in[2^n,2^{n+1}[}\kappa(u)\\
 \ll&\sum_{n=0}^\infty2^nk^{1+\epsilon}\kappa(2^n)\\
 \ll&_{A,j}\sum_{n=0}^\infty k^{4+\epsilon-j\epsilon/2}\left(1+\frac{\det(T)^2}{k^4}\right)^{-A}2^{-n}\\
 \ll&_{A,j}k^{4+\epsilon-j\epsilon/2}\left(1+\frac{\det(T)^2}{k^4}\right)^{-A}.
\end{align*}

We take $j$ big enough to conclude the proof.
\end{proof}

Now, note that for a Heegner point $z=\frac{-\beta+i\sqrt D}\alpha$, $\gamma=\left(\begin{smallmatrix}a&b\\c&d\end{smallmatrix}\right)$ and $z\neq\gamma z$, we have
\begin{align*}
u(z,\gamma z)=&\frac{|z-\gamma z|^2}{4\Im(z)\Im(\gamma z)}=\frac{|z(cz+d)-(az+b)|^2}{4\Im(z)^2}\\
 =&\frac{(c\beta^2-cD-(d-a)\alpha\beta-b\alpha^2)^2+(-2c\beta\sqrt D+(d-a)\alpha\sqrt D)^2}{4\alpha^2 D}.
\end{align*}

In the last line, if the second square is non-zero, we get $u(z,\gamma z)\gg\frac{D}{\alpha^2 D}\gg\frac1{\alpha^2}$. If it is zero, then $2c\beta=(d-a)\alpha$. Moreover, the first square is non-zero, since $z\neq\gamma z$. The first square simplifies then to $(c\beta^2-cD-(d-a)\alpha\beta-b\alpha^2)^2=(-c(\beta^2+D)-b\alpha^2)^2=\alpha^2(-c\delta-b\alpha)^2$. Thus $u(z,\gamma z)\gg\frac{\alpha^2}{\alpha^2D}\gg\frac1D$ in that case. Up to a negligible error, we get that for $z_T\neq\gamma z_T$,
$$u(z_T,\gamma z_T)\gg\min\left(\frac1{\alpha^2},\frac1D\right)\gg k^{-2-\epsilon}.$$

It remains to deal with the $z_T$ and $\gamma$ such that $k^{-2-\epsilon}\ll u\ll k^{-1+\epsilon}$.

\subsection{Orbits of Heegner points}
In this section, we count the $z_T$ and $\gamma$ such that $u(z_T,\gamma z_T)\ll k^{-1+\epsilon}$. This gives the error term in Equation (\ref{equAsymptoteFullSpectrum}). For this, we analyze the distribution of Heegner points and their orbits. Recall that they lie in the classical fundamental domain for a corresponding reduced matrix. We obtain the following result:
\begin{lemma}\label{lemSmalluOrbit}
Let $z_T$ be the Heegner point corresponding to the matrix $T$. Then
$$\sum_{\substack{T\in\mathcal P(\mathbb Z)/\PSL_2(\mathbb Z)\\\det(T)\ll k^{2+\epsilon}}}\frac{\#\Aut(T)}{\epsilon(T)^2\det(T)^{3/2}}\sum_{\substack{\gamma\in\SL_2(\mathbb Z):\\k^{-2-\epsilon}\ll u(z_T,\gamma z_T)\ll k^{-1+\epsilon}}}|\kappa(u(z_T,\gamma z_T))|\ll_\epsilon k^{2.5+\epsilon}.$$
\end{lemma}

\begin{proof}
Since $u(z_T,\gamma z_T)$ is smaller than 1, we have that $\kappa(u(z,\gamma z))\ll k^3$ by Equation (\ref{equBoundKappaTrivial}). Combining this with $\frac{\#\Aut(T)}{\epsilon(T)^2}\ll1$, we see that we only have to show that
$$\sum_{\substack{T\in\mathcal P(\mathbb Z)/\PSL_2(\mathbb Z)\\\det(T)\ll k^{2+\epsilon}}}\frac1{\det(T)^{3/2}}\sum_{\substack{\gamma\in\SL_2(\mathbb Z):\\k^{-2-\epsilon}\ll u(z_T,\gamma z_T)\ll k^{-1+\epsilon}}}1\ll k^{-1/2+\epsilon}$$
to get a bound of size $O(k^{2.5+\epsilon})$ for Equation (\ref{equAsymptoteFullSpectrum}). Geometrically, it is clear that $z_T$ must be close to an edge of the fundamental domain if we want it to be close to another point in its orbit. We make this more precise. There are 2 types of Heegner points such that $u(z_T,\gamma z_T)$ is small. First, suppose $z_T$ has big imaginary part, say $\Im(z_T)>10$. For a translation $\gamma z=z_T+n$, $n\in\mathbb Z$, we have
$$u(z_T,z_T+n)=\frac{n^2}{4\Im(z_T)^2}=\frac{n^2\alpha^2}{4D}.$$

Therefore if $u(z_T,z_T+n)\ll k^{-1+\epsilon}$, we get $n^2\ll\frac D{\alpha^2k^{1-\epsilon}}$. To compute a bound, we sum over $D'=\alpha\delta\asymp D\ll k^{2+\epsilon}$. For a fixed $D'$, there are $\ll (D')^{\epsilon/4}$ choices for $\alpha$ and $\delta$ by the divisor bound and there are $\alpha$ choices for $\beta$, so that
\begin{align*}
\sum_{\substack{T\in\mathcal P(\mathbb Z)/\PSL_2(\mathbb Z)\\\det(T)\ll k^{2+\epsilon}}}\frac1{\det(T)^{3/2}}\sum_{\substack{\gamma\in\SL_2(\mathbb Z)\text{ translation}\\k^{-2-\epsilon}\ll u(z_T,\gamma z_T)\ll k^{-1+\epsilon}}}1\ll&\sum_{D'\ll k^{2+\epsilon}}(D')^{-3/2}\sum_{\alpha\delta=D'}\sum_{|\beta|\leq\alpha/2}\sum_{n^2\ll\frac{D'}{\alpha^2k^{1-\epsilon}}}1\\
 \ll&\sum_{D'\ll k^{2+\epsilon}}(D')^{-3/2+\epsilon/4}\sum_{\alpha\delta=D'}\alpha\sqrt{\frac{D'}{\alpha^2k^{1-\epsilon}}}\\
 \ll_\epsilon&k^{-1/2+\epsilon/2}\sum_{D'\ll k^{2+\epsilon}}(D')^{-1+\epsilon/4}\\
 \ll_\epsilon &k^{-1/2+\epsilon}
\end{align*}

If $\Im(z_T)>10$ and $\gamma$ is not a translation, then $|z_T|\asymp\Im(z_T)$ and $\Im(\gamma z_T)\leq1$. Therefore $|z_T-\gamma z_T|\gg\Im(z_T)-\Im(\gamma z_T)\asymp|z_T|$ and
$$u(z_T,\gamma z_T)=\frac{|z_T-\gamma z_T|^2}{4\Im(z_T)\Im(\gamma z_T)}\gg\frac{|z_T|^2}{\Im(z_T)}\gg|z_T|\gg1.$$

So there is no such $\gamma$ with $u(z_T,\gamma z_T)\ll k^{-1+\epsilon}$. Now, we analyze low-lying Heegner points where $\Im(z_T)\leq10$. Note that we have $|z_T|^2=\frac\delta\alpha\asymp1$. If $\gamma$ is a translation, then the computation above shows that $u(z,\gamma z)\gg1$. Since the Heegner points are in the fundamental domain and we ruled out the case where $\gamma$ is a translation, $\Im(\gamma z_T)\leq1$. Hence
$$u(z_T,\gamma z_T)=\frac{|z_T-\gamma z_T|^2}{4\Im(z_T)\Im(\gamma z_T)}\gg|z_T-\gamma z_T|^2.$$

Any point in the orbit of $z_T$ is at least as far as the point $z_T/|z_T|$ (in the Euclidean distance), as one can see by growing a circle around $z_T$. Suppose that $|z_T|-1\gg k^{(-1+\epsilon)/2}$, then we have
$$|z_T-z_T/|z_T||^2=\frac{|z_T|^2}{|z_T|^2}\left||z_T|-1\right|^2\gg k^{-1+\epsilon},$$
so this rules out this case. Now, if $|z_T|-1\ll k^{(-1+\epsilon)/2}$, that means
$$k^{(-1+\epsilon)/2}\gg|z_T|-1\asymp(|z_T|-1)(|z_T|+1)=|z_T|^2-1=\frac{\delta-\alpha}\alpha.$$

So that $\delta-\alpha\ll k^{(-1+\epsilon)/2}\alpha$. In particular, $\alpha\asymp\delta$ and $\delta^2\asymp\det(T)\ll k^{2+\epsilon}$. Clearly, for such a $z_T$, there is a finite number of $\gamma$ such that $u(z_T,\gamma z_T)\ll k^{-1+\epsilon}$ (at most 12 when $z_T$ is close to $\frac{\pm1+i\sqrt3}2$, for $k$ big enough). Counting such $T$ and $\gamma$ gives us
\begin{align*}
\sum_{\substack{T\in\mathcal P(\mathbb Z)/\PSL_2(\mathbb Z)\\\det(T)\ll k^{2+\epsilon}}}\frac1{\det(T)^{3/2}}\sum_{\substack{\text{such }\gamma\in\SL_2(\mathbb Z)\\k^{-2-\epsilon}\ll u(z_T,\gamma z_T)\ll k^{-1+\epsilon}}}1\ll&\sum_{\delta\ll k^{1+\epsilon}}\delta^{-3}\sum_{\substack{\alpha\\\delta-\alpha\ll k^{(-1+\epsilon)/2}\delta}}\sum_{|\beta|\leq\alpha/2}1\\
 \ll&\sum_{\delta\ll k^{1+\epsilon}}\delta^{-3}\sum_{\substack{\alpha\\\delta-\alpha\ll k^{(-1+\epsilon)/2}\delta}}\alpha\\
 \ll&\sum_{\delta\ll k^{1+\epsilon}}\delta^{-1}k^{(-1+\epsilon)/2}\\
 \ll&k^{-1/2+\epsilon}
\end{align*}
\end{proof}

This concludes the proof of Equation (\ref{equAsymptoteFullSpectrum}). We combine Lemma \ref{lemDecayOfKappa} and Lemma \ref{lemSmalluOrbit} and get the correct error term. For the term $u=0$, the set $\{\gamma\in\SL_2(\mathbb Z)\mid z_T=\gamma z_T\}$ has size $2\epsilon(T)$ since it is its lift from $\PSL_2(\mathbb Z)$.

\subsection{The even spectrum}
In this section, we prove the second equation of Theorem \ref{thmPretraceEstimate}. Let $T_{-1}$ be the $-1$ Hecke operator acting by $T_{-1}\phi(z)=\phi(-\bar z)$. We have
$$\left(\id+T_{-1}\right)\phi(z)=\begin{cases}2\phi(z)&\text{if $\phi$ is even,}\\0&\text{if $\phi$ is odd.}\end{cases}$$

This tells us that
\begin{align*}
\int_{\Lambda_{\ev}}V(\det(T)&,\det(T),t_\phi,k)|\phi(z_T)|^2d\phi\\
 =&\frac14\int_\Lambda V(\det(T),\det(T),t_\phi,k)\left(|\phi(z_T)|^2+|\phi(-\overline{z_T})|^2+2\Re(\phi(z_T)\overline{\phi(-\overline{z_T})})\right)d\phi.
\end{align*}

If $z$ is a Heegner point, then so is $-\overline{z_T}$. So we can consider Equation (\ref{equAsymptoteFullSpectrum}) when we repace $|\phi(z_T)|^2$ by $\phi(z)\overline{\phi(-\bar z)}$. We apply the trace formula again and we consider first the term with $u=0$.

The points $z_T$ and $-\overline{z_T}$ are both in the classical fundamental domain. Therefore if there exists $\gamma\in\SL_2(\mathbb Z)$ such that $\gamma(-\overline{z_T})=z$, that means that $z_T=-\overline{z_T}$ or that $z_T$ is on the edge of the fundamental domain. In all these cases, there is a $\gamma_0$ such that $\gamma_0(-\overline{z_T})=z_T$. This gives 3 possibilities: $\beta=0$ if $z_T=-\overline{z_T}$, and $\beta=-\frac12$ or $|z_T|=1$ otherwise. There $\gamma_0$ is respectively $id$, $\left(\begin{smallmatrix}1&-1\\0&1\end{smallmatrix}\right)$ and $\left(\begin{smallmatrix}0&1\\ -1&0\end{smallmatrix}\right)$. We can post-compose with any $\gamma$ such that $\gamma z_T=z_T$. Therefore the term $\phi(z_T)\phi(-\overline{z_T})$ has the same number of $\gamma$ with $u=0$ as the terms $|\phi(z_T)|^2$ and $|\phi(-\overline{z_T})|^2$. As above, the set $\{\gamma\in\SL_2(\mathbb Z)\mid z_T=\gamma z_T\}$ has size $2\epsilon(T)$. If there is such a $\gamma_0$, we get in total $\frac{8\epsilon(T)}4=2\epsilon(T)$ terms for $u=0$. If there is no $\gamma_0$ such that $\gamma_0(-\overline{z_T})=z_T$, then the term $\phi(z_T)\phi(-\overline{z_T})$ has no term with $u=0$ on the geometric side of the pre-trace. Therefore we only get $\epsilon(T)$. Looking at the table in Appendix \ref{AppEquMatrixAutomorphisms}, we see that the ratio between $\#\Aut(T)$ and $\epsilon(T)$ is 4 if $\gamma_0$ exists and 2 otherwise. Thus we can write this contribution as $\frac{\#\Aut(T)}2$. If we combine this with the factor $\frac{\#\Aut(T)}{\epsilon(T)^2}$ in Equation (\ref{equAsymptoteFullSpectrum}), we get in total
$$\frac12\left(\frac{\#\Aut(T)}{\epsilon(T)}\right)^2,$$
as in the second equation of Theorem \ref{thmPretraceEstimate}.

We consider now the case $u\neq0$. The only thing that matters in the error term of Equation (\ref{equAsymptoteFullSpectrum}) above is the distance $u(z_1,z_2)$ between the two points in the trace formula. If $u(z_T,\gamma(-\overline{z_T}))\gg k^{-1+\epsilon}$, then we conclude as in Lemma \ref{lemDecayOfKappa} using the decay of $\kappa$ and Lemma \ref{lemIwaniecNumberOfgammas}. The other case is when $u(z_T,\gamma(-\overline{z_T}))\ll k^{-1+\epsilon}$. We adapt the proof of Lemma \ref{lemSmalluOrbit}. As before, the Heegner points that are close to the boundary of the fundamental domain contribute. For translations, we have
$$u(z_T,-\overline{z_T}+n)=\frac{|2\Re(z_T)-n|^2}{4\Im(z_T)^2}=\frac{|-2\beta/\alpha-n|^2\alpha^2}{4D}.$$

Since $|\frac\beta\alpha|\leq\frac12$, this is similar to the case above. The only special case is if $n=1$ and $\frac\beta\alpha$ is really close to $-\frac12$. There we have
$$u(z_T,-\overline{z_T}+1)=\frac{(\alpha+2\beta)^2}D.$$

If it is not zero, the square must be at least 1 since $\alpha$ and $2\beta$ are integers. We get $\alpha-2|\beta|\ll\sqrt{Dk^{-1+\epsilon}}$. We compute the number of possible $z_T$ by fixing $D'=\alpha\delta$, getting $\alpha$ via the divisor bound and counting the possibilities for $\beta$. We get
\begin{align*}
\sum_{\substack{T\in\mathcal P(\mathbb Z)/\PSL_2(\mathbb Z)\\k^{-2-\epsilon}\ll u(z_T,-\overline{z_T}+1)\ll k^{-1+\epsilon}}}\frac1{\det(T)^{3/2}}\ll&\sum_{D'\ll k^{2+\epsilon}}(D')^{-3/2+\epsilon/4}\sqrt{D'k^{-1+\epsilon}}\\
 \ll&k^{-1/2+\epsilon/2}\sum_{D'\ll k^{2+\epsilon}}(D')^{-1+\epsilon/4}\ll k^{-1/2+\epsilon}.
\end{align*}

For a low-lying $z_T$, the argument above works since we only considered the distance between $z_T$ and $\frac{z_T}{|z_T|}$. The last case is $z_T$ being close to $-\overline{z_T}$. We have there
$u(z_T,-\overline{z_T})=\frac{\Re(z_T)^2}{\Im(z_T)^2}=\frac{\beta^2}D$, so $\beta^2\ll Dk^{-1+\epsilon}$. We count
\begin{align*}
\sum_{\substack{T\in\mathcal P(\mathbb Z)/\PSL_2(\mathbb Z)\\k^{-2-\epsilon}\ll u(z_T,-\overline{z_T})\ll k^{-1+\epsilon}}}\frac1{\det(T)^{3/2}}\ll&\sum_{D'\ll k^{2+\epsilon}}(D')^{-3/2+\epsilon/4}\sqrt{D'k^{-1+\epsilon}}\ll k^{-1/2+\epsilon}.
\end{align*}

We see that these cases go in the error term too. This concludes the proof of Theorem \ref{thmPretraceEstimate}.

\subsection{Main term of the pre-trace formula}
We now analyze the term with $u=0$ of the pre-trace formula. Applying Theorem \ref{thmPretraceEstimate}, we have
$$N_{\av}^{\rm diag}(K)=\frac{12\pi^2}{\omega K^4}\sum_{k\in2\mathbb N}w\left(\frac kK\right)\sum_{T\in\mathcal P(\mathbb Z)/\PSL_2(\mathbb Z)}\left(\frac{\#\Aut(T)}{\epsilon(T)}\right)^2\frac1{2\det(T)^{3/2}}\kappa(0)+O(K^{-1/2+\epsilon}),$$
with $\kappa(0)$ given in Equation (\ref{equKappa0}). The error term is the combination of Theorem \ref{thmPretraceEstimate} and trivial estimates. In Appendix \ref{AppAutomorphisms}, we calculate all the automorphisms of $T$ in $\GL_2(\mathbb Z)$. At the end of it, a table summarizes the computation. We see that $\frac12(\frac{\#\Aut(T)}{\epsilon(T)})^2$ is 2 except if $T=\left(\begin{smallmatrix}\alpha&\beta\\\beta&\delta\end{smallmatrix}\right)$ is diagonal, $\alpha=\delta$ or $\alpha=2|\beta|$. In these cases the ratio is equal to 8. Recall the definition of $V$ in Equation (\ref{defVfunction}). For $T=Q$, the only part that depends on $T$ is $\det(T)^{-2v}$ for $\Re(v)>0$. We consider
\begin{align*}
\sum_{T\in\mathcal P(\mathbb Z)/\PSL_2(\mathbb Z)}&\left(\frac{\#\Aut(T)}{\epsilon(T)}\right)^2\frac1{2\det(T)^{3/2+2v}}\\
	&=2\sum_{T\in\mathcal P(\mathbb Z)/\PSL_2(\mathbb Z)}\frac1{\det(T)^{3/2+2v}}+6\sum_{\substack{T\in\mathcal P(\mathbb Z)/\PSL_2(\mathbb Z)\\\#\Aut(T)\neq2\epsilon(T)}}\frac1{\det(T)^{3/2+2v}}\\
	&=:2L(v)+6\tilde L(v).
\end{align*}

\begin{lemma}
The function $\tilde L(v)$ converges for $\Re(v)>-1/4$ and is bounded on vertical strips.
\end{lemma}
\begin{proof}
Let $\sigma=\Re(v)$. First, we consider the case of $T$ diagonal. We have
\begin{align*}
\sum_{\substack{T\in\mathcal P(\mathbb Z)/\PSL_2(\mathbb Z)\\T\text{ diagonal}}}\frac1{\det(T)^{3/2+2v}}=\sum_{0<\alpha\leq\delta}\frac1{(\alpha\delta)^{3/2+2v}}\ll\zeta(3/2+2\sigma)^2.
\end{align*}

Therefore this sum converges for all $\sigma>-1/4$ and is bounded on vertical strips. The two other cases are similar:
\begin{align*}
\sum_{\substack{T\in\mathcal P(\mathbb Z)/\PSL_2(\mathbb Z)\\\alpha=\delta}}\frac1{\det(T)^{3/2+2v}}=&\sum_{\substack{0\leq2\beta\leq\delta\\0<\delta}}\frac1{(\delta^2-\beta^2)^{3/2+2v}}\ll\sum_{0<\delta}\frac1{\delta^{2+4\sigma}}\ll\zeta(2+4\sigma),\\
\sum_{\substack{T\in\mathcal P(\mathbb Z)/\PSL_2(\mathbb Z)\\\alpha=2|\beta|}}\frac1{\det(T)^{3/2+2v}}=&\sum_{0<2\beta\leq\delta}\frac1{(2\beta\delta-\beta^2)^{3/2+2v}}\ll\sum_{0<\delta}\frac1{\delta^{3/2+2\sigma}}\sum_{0<\beta}\frac1{\beta^{3/2+2\sigma}}\ll\zeta(3/2+2\sigma)^2.
\end{align*}

Note that these three cases are not disjoint. This is important if one want to estimate the values of $\tilde L$ explicitly.
\end{proof}

To study $L(s)$, we need the following lemma.
\begin{lemma}[\cite{Blomer2019symplectic}, remark after Lemma 12]
Let
$$\tilde h(D):=\#\{T\in\mathcal P(\mathbb Z)/\PSL_2(\mathbb Z)\mid\det(T)=D\}$$
be the class number of the determinant $D$ (corresponding to the discriminant $-4D$)
We have
$$\sum_{D\leq X}\tilde h(D)=\frac{4\pi}9X^{3/2}-X+O(X^{3/4}).$$
\end{lemma}
\begin{remark}
Note that in \cite{Blomer2019symplectic}, we have $\tilde h(D)=h(-4D)$. Hence $X$ must be replaced by $4X$ between the result there and here.
\end{remark}

\begin{lemma}
The function $L(v)$ converges for $\Re(v)>0$ and can be meromorphically extended to $\Re(v)>-1/4$ with a unique pole at $v=0$ of residue $\frac\pi3$. The extension is bounded on vertical strips and away from the pole.
\end{lemma}

\begin{proof}
We have
$$L(s)=\sum_{T\in\mathcal P(\mathbb Z)/\PSL_2(\mathbb Z)}\frac1{\det(T)^{3/2+2v}}=\sum_{0<D}\frac{\tilde h(D)}{D^{3/2+2v}}.$$
This converges for $\Re(v)>0$. Let $X>0$. Summing the Dirichlet series by parts, we get
\begin{align}
\sum_{D\leq X}&\frac{\tilde h(D)}{D^{3/2+2v}}=\sum_{D\leq X}\tilde h(D)X^{-3/2-2v}+(3/2+2v)\int_1^X\sum_{D\leq t}\tilde h(D)\frac{dt}{t^{5/2+2v}}\nonumber\\
	=&\left(\sum_{D\leq X}\tilde h(D)X^{-3/2-2v}-\frac{4\pi}9X^{-2v}\right)+(3/2+2v)\int_1^X\left(\sum_{D\leq t}\tilde h(D)-\frac{4\pi}9t^{3/2}\right)\frac{dt}{t^{5/2+2v}}\nonumber\\
	&+\frac{4\pi}9X^{-2v}+(3/2+2v)\int_1^X\frac{4\pi}9t^{-1-2v}dt.
\end{align}

The last integral is
$$(3/2+2v)\int_1^X\frac{4\pi}9t^{-1-2v}dt=-(3/2+2v)\frac{4\pi}9\frac{X^{-2v}-1}{2v}.$$

For $\Re(v)>0$, the limit as $X\to\infty$ converges to $\frac{4\pi}9\frac{3/2+2v}{2v}$. Finally,
$$\lim_{X\to\infty}\left(\sum_{D\leq X}\tilde h(D)X^{-3/2-2v}-\frac{4\pi}9X^{-2v}\right)=\lim_{X\to\infty}(X^{-1/2-2v}+O(X^{-3/4-2v}))=0$$

In total, we have that
\begin{align*}
(3/2+2v)\int_1^\infty\left(\sum_{D\leq t}\tilde h(D)-\frac{4\pi}9t^{3/2}\right)\frac{dt}{t^{5/2+2v}}+\frac{4\pi}9\frac{3/2+2v}{2v}
\end{align*}
converges for $\Re(v)>-\frac14$ and $v\neq0$. It is an meromorphic continuation of $L(v)$ with a unique pole of residue $\Res_{v=0}L(v)=\frac\pi3$ and it is bounded on vertical strips and away from $v=0$.
\end{proof}

Now, we consider the $T$-sum combined with the $v$-integral of Equation (\ref{defVfunction}). 
\begin{align*}
\sum_{T\in\mathcal P(\mathbb Z)/\PSL_2(\mathbb Z)}&\left(\frac{\#\Aut(T)}{\epsilon(T)}\right)^2\frac1{2\det(T)^{3/2}}\\
	&\quad\cdot\frac1{2\pi i}\int_{(3)}e^{v^2}c_k^{-1}G(\tau,k,v+1/2+it)G(\tau,k,v+1/2-it)\det(T)^{-2v}\frac{dv}v\\
	&=\frac1{2\pi i}\int_{(3)}e^{v^2}c_k^{-1}G(\tau,k,v+1/2+it)G(\tau,k,v+1/2-it)(2L(v)+6\tilde L(v))\frac{dv}v.
\end{align*}
Note that the integrand has a double pole at $v=0$. We have the following Taylor expansion for the gamma factor:
\begin{align*}
c_k^{-1}G(\tau,k,v+1/2+it)&G(\tau,k,v+1/2-it)=c_k^{-1}G(\tau,k,1/2+it)G(\tau,k,1/2-it)\\
 &\left[1+v\left(\sum_{\pm\pm}\frac{\Gamma'}{\Gamma}\left(\frac{k-1/2}2\pm it\pm\frac{i\tau}2\right)-4\log(2\pi)\right)+O(v^2)\right].
\end{align*}

Recall that according to Equation (\ref{eqcutoff}), the $t$ and $\tau$-integral can be cut at $k^{1/2+\epsilon}$ up to a negligible error. Moreover, $\frac{\Gamma'}\Gamma(z)=\log(z)+O(|z|^{-1})$ so that for $t,\tau\ll k^{1/2+\epsilon}$,
\begin{align*}
\sum_{\pm\pm}\frac{\Gamma'}{\Gamma}\left(\frac{k-1/2}2\pm it\pm\frac{i\tau}2\right)-4\log(2\pi)&=\sum_{\pm\pm}\log((k-1/2)/2\pm it\pm i\tau/2)-4\log(2\pi)+O(k^{-1})\\
 &=4\log(k)+C_0+O(k^{-1/2+\epsilon}),
\end{align*}
for some constant $C_0\in\mathbb R$. Let $C_1=\lim_{v\to0}(L(v)-\frac\pi{3v})$ be the constant term of the Laurent series of $L(v)$. In conclusion, the pole at $v=0$ of the integrand has residue
$$c_k^{-1}G(\tau,k,1/2+it)G(\tau,k,1/2-it)\left(\frac{2\pi}3\left(4\log(k)+C_0+O(k^{-1/2+\epsilon})\right)+2C_1+6\tilde L(0)\right).$$
We define $D=\frac{2\pi}3C_0+2C_1+6\tilde L(0)$. We move the $v$-integral to $\Re(v)=-1/4+\epsilon$ for some fixed $\epsilon>0$:
\begin{align*}
\frac1{2\pi i}\int_{(3)}e^{v^2}&c_k^{-1}G(\tau,k,v+1/2+it)G(\tau,k,v+1/2-it)(2L(v)+6\tilde L(v))\frac{dv}v\\
 &=c_k^{-1}G(\tau,k,1/2+it)G(\tau,k,1/2-it)\left(\frac{8\pi}3\log(k)+D+O(k^{-1/2+\epsilon})\right)\\
 &\quad+\frac1{2\pi i}\int_{(-1/4+\epsilon)}e^{v^2}c_k^{-1}G(\tau,k,v+1/2+it)G(\tau,k,v+1/2-it)(2L(v)+6\tilde L(v))\frac{dv}v
\end{align*}

We apply the bounds of Lemma \ref{lemBoundGammaFactors} to the second term to get
\begin{align*}
\frac1{2\pi i}&\int_{(-1/4+\epsilon)}e^{v^2}c_k^{-1}G(\tau,k,v+1/2+it)G(\tau,k,v+1/2-it)(2L(v)+6\tilde L(v))\frac{dv}v\\
 \ll_A&k^{1/2+4\epsilon}\int_{-\infty}^\infty e^{-w^2}\left(1+\frac{t^2+|\tau|^2+w^2}k\right)^{-A}(|L(-1/4+\epsilon+iw)|+|\tilde L(-1/4+\epsilon+iw)|)dw\\
 \ll_A&k^{1/2+4\epsilon}\left(1+\frac{t^2+|\tau|^2}k\right)^{-A}.
\end{align*}

Using Lemma \ref{lemBoundGammaFactors} and $\tau\tanh(\pi\tau)=|\tau|+O(1)$, we get
\begin{align*}
\frac{12\pi}{\omega K^4}\sum_{k\in2\mathbb N}&w\left(\frac kK\right)\int_{-\infty}^\infty\int_{-\infty}^\infty\frac1{2\pi i}\int_{(-1/4+\epsilon)}e^{v^2}\\
 &\cdot c_k^{-1}G(\tau,k,v+1/2+it)G(\tau,k,v+1/2-it)(2L(v)+6\tilde L(v))\frac{dv}vdt\tau\tanh(\pi\tau)d\tau\\
 \ll_A&K^{-3}\int_{-\infty}^\infty\int_{-\infty}^\infty K^{1/2+4\epsilon}\left(1+\frac{t^2+|\tau|^2}k\right)^{-A}dt\tau\tanh(\pi\tau)d\tau\\
 \ll&K^{-3}\cdot K^{1/2+4\epsilon}\cdot K^{1/2}\cdot K\ll K^{-1+4\epsilon}.
\end{align*}

Therefore we conclude that
\begin{align*}
N_{\av}^{\rm diag}(K)=&\frac{12\pi^2}{\omega K^4}\sum_{k\in2\mathbb N}w\left(\frac kK\right)\frac1{4\pi}\int_{-\infty}^\infty\int_{-\infty}^\infty c_k^{-1}G(\tau,k,1/2+it)G(\tau,k,1/2-it)\\
	&\quad\cdot\left(\frac{8\pi}3\log(k)+D+O(k^{-1/2+\epsilon})\right)dt\tau\tanh(\tau)d\tau+O(K^{-1/2+\epsilon})\\
	&=\frac{3\pi}{\omega K^4}\sum_{k\in2\mathbb N}w\left(\frac kK\right)\int_{-\infty}^\infty\int_{-\infty}^\infty c_k^{-1}G(\tau,k,1/2+it)G(\tau,k,1/2-it)dt\tau\tanh(\tau)d\tau\\
	&\quad\cdot\left(\frac{8\pi}3\log(k)+D+O(k^{-1/2+\epsilon})\right)+O(K^{-1/2+\epsilon}).
\end{align*}

Now, we compute an approximation of the $t$-integral using Lemma \ref{lemBoundGammaFactors}. We also replace the gamma factors outside $t,\tau\ll k^{1/2+\epsilon}$. This gives an error of size $O_A(k^{-A})$ for all $A>0$, so it is negligible. We get
\begin{align*}
\int_{-\infty}^\infty c_k^{-1}G(\tau,k,1/2+it)G(\tau,k,1/2-it)dt&=\frac2{\pi^{5/2}}k^{3/2}\int_{-\infty}^\infty\exp\left(-\frac{4t^2+|\tau|^2}k\right)\left(1+O(k^{-1/2+\epsilon})\right)dt\\
	&\quad+O_A(k^{-A})\\
	&=\frac2{\pi^{5/2}}k^{3/2}\frac{\sqrt{\pi k}}2\exp\left(-\frac{|\tau|^2}k\right)\left(1+O(k^{-1/2+\epsilon})\right)\\
	&=\frac1{\pi^2}k^2\exp\left(-\frac{|\tau|^2}k\right)+O(k^{3/2+\epsilon})
\end{align*}

We compute the $\tau$-integral using $\tau\tanh(\tau)=|\tau|+O(1)$. This gives
\begin{align*}
\int_{-\infty}^\infty\exp(-\tau^2/k)\tau\tanh(\tau)d\tau&=2\int_0^\infty\exp(-\tau^2/k)(\tau+O(1))d\tau\\
	&=-\left.k\exp(-\tau^2/k)\right|_0^\infty+O(\sqrt k)\\
	&=k+O(\sqrt k).
\end{align*}

We conclude that
\begin{align*}
\int_{-\infty}^\infty\int_{-\infty}^\infty c_k^{-1}G(\tau,k,1/2+it)G(\tau,k,1/2-it)dt\tau\tanh(\tau)d\tau&=\frac1{\pi^2}k^2(k+O(\sqrt k))+O(k^{3/2+\epsilon})\\
	&=\frac1{\pi^2}k^3+O(k^{2.5+\epsilon}).
\end{align*}

\subsection{Sum over $k$}
Recall that $\omega=\int_1^2w(x)x^3dx$. We saw above that the diagonal term is
\begin{align*}
N_{\av}^{\rm diag}(f)&=\frac{3\pi}{\omega K^4}\sum_{k\in2\mathbb N}w\left(\frac kK\right)\frac1{\pi^2}k^3\left(\frac{8\pi}3\log(k)+D+O(k^{-1/2+\epsilon})\right)+O(K^{-1/2+\epsilon})\\
 &=\frac8{\omega K^4}\sum_{k\in2\mathbb N}w\left(\frac kK\right)k^3\log(k)+\frac{3D}{\omega\pi K^4}\sum_{k\in2\mathbb N}w\left(\frac kK\right)k^3+O(K^{-1/2+\epsilon}).
\end{align*}

We deduce the main term of Theorem \ref{thmNav} by summing over $k$. We apply the Euler-MacLaurin formula for this.
\begin{lemma}[\cite{IwaniecKowalski}, Lemma 4.1]
Let $a,b\in\mathbb Z$ and $f$ a $C^1$ function on $[a,b]$. Then
$$\sum_{\substack{n\in2\mathbb N\\a\leq n\leq b}}f(n)=\frac12\int_a^bf(x)dx+O\left(\int_a^b|f'(x)|dx+|f(a)|+|f(b)|\right).$$
\end{lemma}

Let $\omega'=\int_1^2w(x)x^3dx$. We get
\begin{align*}
N_{\rm av}^{\rm diag}(f)&=\frac4{\omega K^4}\int_K^{2K}w\left(\frac xK\right)x^3\log(x)dx+\frac{3D}{2\omega\pi K^4}\int_K^{2K}w\left(\frac xK\right)x^3dx+O(K^{-1/2+\epsilon})\\
	&\quad+O\left(\frac1{K^4}\int_K^{2K}\left(\frac1K w'\left(\frac xK\right)x^3\log(x)+w\left(\frac xK\right)x^2\log(x)+w\left(\frac xK\right)x^2\right)dx\right)\\
	&=\frac4{\omega K^3}\int_1^2w(x)(xK)^3\log(xK)dx+\frac{3D}{2\omega\pi K^3}\int_1^2w(x)(xK)^3dx+O(K^{-1/2+\epsilon})+O(K^{-1+\epsilon})\\
	&=4\log(K)+4\frac{\omega'}{\omega}+\frac{3D}{2\pi}+O(K^{-1/2+\epsilon})\\
	&=4\log(K)+D'+O(K^{1/2+\epsilon}).
\end{align*}
Here $D'$ is a constant that only depend on $w$, $\epsilon>0$ is arbitrary and the implied constant depends only on $\epsilon$ and $w$.

\section{Rank 1 term}\label{section5}
We focus now on the first non-diagonal term of the Kitaoka formula, called the rank 1 term. It comes from the combination of Equations (\ref{defNav}), (\ref{equNormPreKitaoka}) and the Kitaoka formula (Theorem \ref{thmKitaokaFormula}). Its shape is
\begin{align*}
\frac{12\sqrt2\pi^3}{\omega K^4}\sum_{k\in2\mathbb N}&w\left(\frac kK\right)\sum_{T,Q\in\mathcal P(\mathbb Z)/\PSL_2(\mathbb Z)}\frac1{\epsilon(T)\epsilon(Q)\det(TQ)^{3/4}}\\
	&\cdot\int_{\Lambda_{\ev}}V(\det(T),\det(Q),t_\phi,k)\phi(z_T)\bar\phi(z_Q)d\phi\\
	&\cdot\sum_\pm\sum_{c,s\geq1}\sum_{U,V}\frac{(-1)^{k/2}}{c^{3/2}s^{1/2}}H^\pm(UQU^t,V^{-1}TV^{-t};c)J_\ell\left(\frac{4\pi\sqrt{\det(TQ)}}{cs}\right).
\end{align*}

We have various sums that we need to restrict, up to a negligible error. First, we apply Lemma \ref{lemSumK}. The sum over $k$ is
$$\sum_{k\in2\mathbb N}w\left(\frac kK\right)(-1)^{k/2}V(\det(T)\det(Q),t_\phi,k)J_\ell\left(\frac{4\pi\sqrt{\det(TQ)}}{cs}\right).$$

We get three terms. The $w_0$ term is negligible because all the other sums and integral have a cut-off that gives a polynomial growth in $K$. The terms with $w_+$ and $w_-$ have the property that $w_\pm(x)\ll_AK^2\left(1+\frac{K^2}x\right)^{-A}$ with $x=\frac{4\pi\sqrt{\det(TQ)}}{cs}$. They also depend on $\det(T)$, $\det(Q)$ and $t_\phi$ and follow the other bounds of Equation (\ref{eqcutoff}). In our case, we have
$$w_\pm(x,x_1,x_2,\tau,K)\ll_AK^2\left(1+\frac{x_1x_2}{K^4}\right)^{-A}\left(1+K^{1/2}|\log(x_2/x_1)|\right)^{-A}\left(1+\frac{|\tau|^2}K\right)^{-A}\left(1+\frac{K^2}{x}\right)^{-A}.$$
with $x$ as above, $x_1=\det(T)$, $x_2=\det(Q)$, $\tau=t_\phi$. We also have a control on the derivatives given by Equation (\ref{eqcutoff}) and Lemma \ref{lemSumK}.

\subsection{First upper bound}
We prove a first easy bound for the rank 1 term. Let $\epsilon>0$. We may change the value of $\epsilon$ when we refer to older computations. Combining the estimate of Lemma \ref{lemSumK} with Equation (\ref{eqcutoff}), we get $c^2s^2K^{4-\epsilon}\ll\det(TQ)\ll K^{4+\epsilon}$ up to a negligible error. Hence $c,s=O(K^{\epsilon})$ and $K^{4-\epsilon}\ll\det(TQ)\ll K^{4+\epsilon}$. Since $\det(T)-\det(Q)\ll K^{-1/2+\epsilon}\det(T)$, we get $K^{2-\epsilon}\ll\det(T),\det(Q)\ll K^{2+\epsilon}$. Now, we look at the exponential sum $H^\pm$ defined after Theorem \ref{thmKitaokaFormula}. It vanishes unless there are $U=\left(\begin{smallmatrix}*&*\\u_3&u_4\end{smallmatrix}\right)/\{\pm1\}$ and $V=\left(\begin{smallmatrix}v_1&*\\v_3&*\end{smallmatrix}\right)$ in $\GL_2(\mathbb Z)$ such that
$$(UQU^t)_{22}=(V^{-1}TV^{-t})_{22}=s.$$

Let $T=\left(\begin{smallmatrix}a&b\\b&c\end{smallmatrix}\right)$, $Q=\left(\begin{smallmatrix}x&y\\y&z\end{smallmatrix}\right)$. Using the inequality $r^2+t^2\geq2rt$, this gives
\begin{align*}
s&=av_3^2-2bv_1v_3+cv_1^2\geq2(\sqrt{ac}-|b|)|v_3v_4|,\\
s&=xu_3^2+2yu_3u_4+zu_4^2\geq2(\sqrt{xz}-|y|)|u_3u_4|.
\end{align*}

Since $T$ and $Q$ are reduced, we have $2(\sqrt{ac}-|b|)\geq\sqrt{ac}\geq\sqrt{\det(T)}$ and similarly for $Q$. If $u_3u_4$ or $v_1v_3$ is non-zero, then we get $s\geq\sqrt{\det(T)}$ or $\sqrt{\det(Q)}$ and both are of size $\gg K^{1-\epsilon}$. Since $s=O(K^\epsilon)$ up to a negligible error, this is negligible. Otherwise we have $u_4=v_1=0$ because $c\gg\sqrt{\det(T)}\gg K^{1-\epsilon}$ and similarly $z\gg K^{1-\epsilon}$. Since $U,V\in\GL_2(\mathbb Z)$, we have the following choices of representatives for $U$ and $V$:
\begin{align}\label{equUV}
U&=\begin{pmatrix}0&1\\1&0\end{pmatrix},&
V&=\pm\begin{pmatrix}0&1\\1&0\end{pmatrix}.
\end{align}

We get $s=a=x$ and in particular $x,a=O(K^\epsilon)$. Since $T$ and $Q$ are reduced, we also have $|y|,|b|=O(K^\epsilon)$ and $K^{2-\epsilon}\ll z\asymp c\ll K^{2+\epsilon}$. Therefore there are $O(K^{2+\epsilon}\cdot K^{3/2+\epsilon})$ choices for $T$ and $Q$ and $O(K^\epsilon)$ choices for $c,s$ and $U,V$. Combining with other estimates (recall that the exponential sum is bounded by $c^2$) and Equation (\ref{lemBoundNonDiagSpect}) with $T=K^{1/2+\epsilon}$ and $K^{1-\epsilon}\ll\Im(z_T),\Im(z_Q)\ll K^{1+\epsilon}$, we get that the rank 1 term is bounded by
\begin{align*}
K^{-4}&\sum_{T,Q\in\mathcal P(\mathbb Z)/\PSL_2(\mathbb Z)}\frac1{\epsilon(T)\epsilon(Q)\det(TQ)^{3/4}}\sum_{c,s\geq1}\sum_\pm\sum_{U,V}\frac{(-1)^{k/2}}{c^{3/2}s^{1/2}}H^\pm(UQU^t,V^{-1}TV^{-t};c)\\
	&\quad\cdot e\left(\pm\frac{2\sqrt{\det(TQ)}}{cs}\right)\int_{\Lambda_{\ev}}w_\pm\left(\frac{4\pi\sqrt{\det(TQ)}}{cs},\det(T),\det(Q),t_\phi,K\right)\phi(z_T)\bar\phi(z_Q)d\phi\\
	&\ll K^{-4}\cdot K^{3.5+\epsilon}\cdot K^{-3+\epsilon}\cdot K^\epsilon\cdot K^2\cdot K^{3/2+\epsilon}\\
	&\ll K^{4\epsilon}.
\end{align*}

\subsection{Analysis of the $T,Q$-sum}
We need to win extra cancellation somewhere. We do that in the $T,Q$-sum. We consider $\Delta=\det(Q)-\det(T)$. We know that $\Delta=O(K^{3/2+\epsilon})$ up to a negligible error. We can fix all the coefficients of $Q$ except $z$ at the cost of $K^\epsilon$ choices. The possible values of $\Delta=xz-y^2-\det(T)$ follow then an arithmetic progression as $z$ varies. More precisely, $d:=y^2+\det(T)\equiv\Delta\bmod x$. Looking at the last table in Appendix \ref{AppAutomorphisms}, we have $\epsilon(Q)=1$ unless $x=z$, which is a negligible case for $K$ big enough. Similarly, we can suppose that $\epsilon(T)=1$. The $T,Q$-sum looks like
\begin{align}\label{equRank1TQSum}
&\sum_{K^{2-\epsilon}\ll\det(T)\ll K^{2+\epsilon}}\sum_{\substack{|\Delta|\ll O(K^{3/2+\epsilon})\\\Delta\equiv d\bmod x}}\frac1{(\det(T)(\det(T)+\Delta))^{3/4}}\\
	&\quad\cdot\int_{\Lambda_{\ev}}w_\pm\left(\frac{4\pi\sqrt{\det(T)(\det(T)+\Delta)}}{cs},\det(T),\det(T)+\Delta,t_\phi,K\right)\phi(z_T)\bar\phi(z_Q)d\phi\nonumber\\
	&\quad\cdot H^\pm(UQU^t,V^{-1}TV^{-t};c)e\left(\pm\frac{2\sqrt{\det(T)(\det(T)+\Delta)}}{cs}\right).\nonumber
\end{align}

Recall the definition of $H^\pm$ just after Theorem \ref{thmKitaokaFormula}, the representatives of $U$ and $V$ chosen in equation (\ref{equUV}) and that $s=a=x$. We get that
\begin{align*}
P=UQU^t&=\begin{pmatrix}z&y\\y&s\end{pmatrix},&
S=V^{-1}TV^{-t}&=\begin{pmatrix}c&b\\b&s\end{pmatrix}.
\end{align*}
Therefore, $z=p_1$ and the summand in $H^\pm$ is
\begin{align}
e&\left(\frac{\bar d_1s_4d_2^2\mp\bar d_1p_2d_2+s_2d_2+\bar d_1p_1+d_1s_1}c\mp\frac{p_2s_2}{2cs_4}\right)\nonumber\\
	&\quad=e\left(\frac{\bar d_1z}c\right)e\left(\frac{\bar d_1s_4d_2^2\mp\bar d_1p_2d_2+s_2d_2+d_1s_1}c\mp\frac{p_2s_2}{2cs_4}\right)\nonumber\\
	&\quad=e\left(\frac{\bar d_1\Delta}{cs}\right)e\left(\frac{\bar d_1(\det(T)+y^2)}{cs}\right)e\left(\frac{\bar d_1s_4d_2^2\mp\bar d_1p_2d_2+s_2d_2+d_1s_1}c\mp\frac{p_2s_2}{2cs_4}\right).\label{equDeltainHpm}
\end{align}

We fix $\Delta\bmod cs$, so that we can see this term as constant in the $\Delta$-sum. This adds a sum over $d\bmod cs$ such that $d\equiv y^2+\det(T)\bmod s$. Now, we consider the spectral integral.

\begin{lemma}\label{lemwtilde}
We have
\begin{align*}
\int_{\Lambda_{\ev}}w_\pm&\left(\frac{4\pi\sqrt{\det(T)(\det(T)+\Delta)}}{cs},\det(T),\det(T)+\Delta,t_\phi,K\right)\phi(z_T)\bar\phi(z_Q)d\phi\\
	&=\tilde w_\pm\left(\frac{4\pi\sqrt{\det(T)(\det(T)+\Delta)}}{cs},\det(T),\det(T)+\Delta,K\right)+O(K^{2.5+\epsilon}),
\end{align*}
where the function
\begin{align*}
w_\pm(x,x_1,x_2,k)&:=\frac{(x_1x_2)^{1/4}}s\int_{-\infty}^\infty w_\pm\left(x,x_1,x_2,\tau,k\right)\\
	&\quad\cdot\left(\left(\frac{x_1}{x_2}\right)^{i\tau}+\left(\frac{x_2}{x_1}\right)^{i\tau}+\nu(1/2-i\tau)(x_1x_2)^{i\tau}+\nu(1/2+i\tau)(x_1x_2)^{-i\tau}\right)\frac{d\tau}{4\pi}
\end{align*}
(for $s$ fixed) satisfies the following bounds:
\begin{align*}
x^{j_1}\left(\frac{x_1}{k^{1/2}}\right)^{j_2}&\left(\frac{x_1}{k^{1/2}}\right)^{j_3}\frac{d^{j_1}}{dx^{j_1}}\frac{d^{j_2}}{dx_1^{j_2}}\frac{d^{j_3}}{dx_2^{j_3}}\tilde w_\pm(x,x_1,x_2,k)\\
	&\ll_{A,j_1,j_2,j_3}k^{2.5}(x_1x_2)^{1/4}\left(1+\frac{k^2}x\right)^{-A}\left(1+\frac{x_1x_2}{k^2}\right)^{-A}\left(1+\frac{(x_1-x_2)k^{1/2}}{x_1}\right)^{-A}.
\end{align*}
\end{lemma}

\begin{proof}
We have $\Im(z_T),\Im(z_Q)\gg K^{1-\epsilon}$ and the spectral parameter $t_\phi$ satisfies $|t_\phi|\ll K^{1/2+\epsilon}$ up to a negligible error. In that case, the cusps forms in the spectral decomposition are known to be negligible and the only terms that remain are the constant terms in the Fourier expansion of the Eisenstein series and the constant function. Details about the decay of the $K$-Bessel function and the Eisenstein series can be found in Lemma 3.1 of \cite{Young2018}. For the Fourier coefficients of cusp forms, a polynomial bound like Equation 8.8 in \cite{IwaniecSpectralMethods} suffices. More precisely, let
\begin{align*}
y_1&=\frac{\sqrt{\det(T)}}s,	&y_2&=\frac{\sqrt{\det(Q)}}s
\end{align*}
and
$$\nu(s)=\pi^{1/2}\frac{\Gamma(s-1/2)}{\Gamma(s)}\frac{\zeta(2s-1)}{\zeta(2s)}=\frac{\pi^{-(1-s)}\Gamma(1-s)\zeta(2(1-s))}{\pi^{-s}\Gamma(s)\zeta(2s)}.$$
Then $|\nu(s)|=1$ and the constant term of the Eisenstein series $E(x+iy,s)$ is $y^s+\nu(1-s)y^{1-s}$. We have
\begin{align*}
\int_{\Lambda_{\ev}}w_\pm&\left(\frac{4\pi\sqrt{\det(T)(\det(T)+\Delta)}}{cs},\det(T),\det(T)+\Delta,t_\phi,K\right)\phi(z_T)\bar\phi(z_Q)d\phi\\
	&=\int_{-\infty}^\infty w_\pm\left(\frac{4\pi\sqrt{\det(T)(\det(T)+\Delta)}}{cs},\det(T),\det(T)+\Delta,t_\phi,K\right)\\
	&\quad\cdot\left(y_1^{1/2+i\tau}+\nu(1/2+i\tau)y_1^{1/2-i\tau}\right)\left(y_2^{1/2-i\tau}+\nu(1/2-i\tau)y_2^{1/2+i\tau}\right)\frac{d\tau}{4\pi}\\
	&\quad+\int_{-\infty}^\infty V(\tau)\frac3\pi d\tau+O(e^{-cK})\\
&=\tilde w_\pm\left(\frac{4\pi\sqrt{\det(T)(\det(T)+\Delta)}}{cs},\det(T),\det(T)+\Delta,K\right)+O(K^{2.5+\epsilon}).
\end{align*}

Note that the bounds and control on the derivatives in the other variables of $w_\pm$ also apply to $\tilde w_\pm$. We use $|\tau|\ll K^{1/2+\epsilon}$, $s\geq1$ and $|\nu(1/2\pm i\tau)|=1$ to get the stated bound.
\end{proof}

\begin{remark}
In the definition of $\tilde w_\pm$, it is possible to integrate by parts the $\tau$-integral multiple times for the factors with the terms $(y_1/y_2)^{\pm it\tau}$. It gives a cut-off of the form $(k^{1/2}\log(y_2/y_1))^{-j}$ and we can get a strong decay for the other terms. We saw that $a=x=s$. Hence this is redundant information with the cut-off on $\det(Q)/\det(T)$ of Equation \ref{eqcutoff} in our case.
\end{remark}

We also have
$$\frac1{(\det(T)+\Delta)^{3/4}}=\frac1{\det(T)^{3/4}}+O(K^{-2+\epsilon}),$$
up to a negligible error. Inserting this, Equation (\ref{equDeltainHpm}) and the result of Lemma \ref{lemwtilde} in Equation (\ref{equRank1TQSum}), we get
\begin{align*}
\sum_{\det(T)\asymp K^2}&\det(T)^{-3/2}\sum_{\substack{d\bmod cs\\d\equiv y^2+\det(T)\bmod s}}H^\pm(P,S,c)\sum_{\substack{|\Delta|\ll O(K^{3/2})\\\Delta\equiv d\bmod cs}}e\left(\pm\frac{2\sqrt{\det(T)(\det(T)+\Delta)}}{cs}\right)\\
	&\cdot\tilde w_\pm\left(\frac{4\pi\sqrt{\det(T)(\det(T)+\Delta)}}{cs},\det(T),\det(T)+\Delta,k\right)+O(K^{3+\epsilon}).
\end{align*}

\subsection{Poisson summation and stationary phase}
We apply Poisson summation formula to the $\Delta$-sum.
\begin{align*}
\sum_{\substack{\Delta=O(K^{3/2})\\\Delta=d\bmod{cs}}}&\tilde w_\pm\left(\frac{4\pi\sqrt{\det(T)(\det(T)+\Delta)}}{cs},\det(T),\det(T)+\Delta,K\right)e\left(\pm\frac{2\sqrt{\det(T)(\det(T)+\Delta)}}{cs}\right)\\
&=\frac1{cs}\sum_{h\in\mathbb Z}\int_{-\infty}^\infty\tilde w_\pm\left(\frac{4\pi\sqrt{\det(T)(\det(T)+t)}}{cs},\det(T),\det(T)+t,K\right)\\
	&\quad\cdot e\left(\pm\frac{2\sqrt{\det(T)(\det(T)+t)}+h(d-t)}{cs}\right)dt.
\end{align*}

To analyze this integral, we need to compute the derivative of $\tilde w_\pm$ with respect to $t$. We have
\begin{align*}
\frac d{dt}\tilde w_\pm&\left(\frac{4\pi\sqrt{\det(T)(\det(T)+t)}}{cs},\det(T),\det(T)+t,K\right)\\
	&=\left(\frac d{dx}\tilde w_\pm\right)\left(\frac{4\pi\sqrt{\det(T)(\det(T)+t)}}{cs},\det(T),\det(T)+t,K\right)\frac{4\pi}{cs}\sqrt{\frac{\det(T)}{\det(T)+t}}\\
	&\quad+\left(\frac d{dx_2}\tilde w_\pm\right)\left(\frac{4\pi\sqrt{\det(T)(\det(T)+t)}}{cs},\det(T),\det(T)+t,K\right)\\
	&\ll_A\left(K^{-2+\epsilon}+K^{-3/2+\epsilon}\right)(\det(T)(\det(T)+t))^{1/4}K^{2.5}\left(1+\frac{csK^2}{\sqrt{\det(T)(\det(T)+t)}}\right)^{-A}\\
	&\quad\cdot\left(1+\frac{\det(T)(\det(T)+t)}{K^2}\right)^{-A}\left(1+\frac{tK^{1/2}}{\det(T)}\right)^{-A}
\end{align*}

More generally, each derivative with respect to $t$ adds a factor of size $K^{-3/2+\epsilon}$ (up to a constant depending on $j$). This is because it either adds a derivative in the first or the third variable of $\tilde w_\pm$, or it differentiates a factor of the form $(\det(T)+t)^{-r}$. All these added factors are of size $\ll K^{-3/2+\epsilon}$. If $|h|\gg K^\epsilon$, we integrate by parts, until we can sum over $h$ and get a big enough power saving. Each derivative in $t$ adds in the worst case nothing here. But the $h$-sum can be a small as we want, so we get a strong decay. More precisely,
\begin{align*}
\sum_{|h|\gg K^\epsilon}&e\left(\frac{hd}{cs}\right)\int_{-\infty}^\infty\tilde w_\pm\left(\frac{4\pi\sqrt{\det(T)(\det(T)+t)}}{cs},\det(T),\det(T)+t,k\right)\\
	&\quad\cdot e\left(\pm\frac{2\sqrt{\det(T)(\det(T)+t)}-ht}{cs}\right)dt\\
	&=\sum_{|h|\gg K^\epsilon}e\left(\frac{hd}{cs}\right)\left(\frac{cs}{-2\pi ih}\right)^j\int_{-\infty}^\infty\frac{d^j}{dt^j}\left[\tilde w_\pm\left(\frac{4\pi\sqrt{\det(T)(\det(T)+t)}}{cs},\det(T),\det(T)+t,k\right)\right.\\
	&\quad\cdot\left.e\left(\pm\frac{2\sqrt{\det(T)(\det(T)+t)}}{cs}\right)\right]e\left(-\frac{ht}{cs}\right)dt\\
	&\ll_{A,j,\epsilon}K^{1+\epsilon}\sum_{|h|\gg K^\epsilon}\frac1{h^j}\int_{-\infty}^\infty K^{3.5}\left(1+\frac{tK^{1/2}}{x_1}\right)^{-A}dt\left(1+\frac{csK^2}{\det(T)}\right)^{-A}\left(1+\frac{\det(T)^2}{K^2}\right)^{-A}\\
	&\ll_A K^{-A}\left(1+\frac{csK^2}{\det(T)}\right)^{-A}\left(1+\frac{\det(T)^2}{K^2}\right)^{-A}.
\end{align*}

Using the cut-off on the other sums, we see that this term is negligible. For small $h$, we apply the stationary phase method. The stationary point is
\begin{align*}
\pm\frac1{cs}\sqrt{\frac{\det(T)}{\det(T)+t_0}}=\frac h{cs}\Longrightarrow t_0=\frac{\det(T)}{h^2}-\det(T).
\end{align*}
Note that $h$ must have the same sign as the left-hand side. There are three cases. If $h=0$, then there is no stationary point. We apply in that case Lemma \ref{lemBKY8.1}. If $h=\pm1$, then $t_0=0$. We apply Lemma \ref{lemBKY8.2}. Otherwise, $t_0\gg\det(T)K^{-\epsilon}$ and $\tilde w_\pm$ is negligible for such $t$. We apply again Lemma \ref{lemBKY8.1}. Following notations there, we have $w=\tilde w_\pm$ and
$$h(t)=2\pi\left(\pm\frac{2\sqrt{\det(T)(\det(T)+t)}+h(d-t)}{cs}\right).$$

In the first and the last case, we get
\begin{align*}
\alpha&=-K^{3/2+\epsilon}, &\beta&=K^{3/2+\epsilon},\\
X&=K^{3.5+\epsilon}, &U&=K^{1.5},\\
R&=K^{-\epsilon},\\
Y&=K^{2+\epsilon}, &Q&=K^{2-\epsilon}.
\end{align*}

Lemma \ref{lemBKY8.1} tells us that the integral is bounded by
$$\ll_A K^{3/2+\epsilon}\cdot K^{3.5+\epsilon}[(K^{2-2\epsilon}/K^{1+\epsilon})^{-A}+K^{-1.5A}].$$
Using the cut-off on the other sums, we see that these terms are negligible. If $t_0=0$, we apply Lemma \ref{lemBKY8.2}. Following the notations, we get
\begin{align*}
\alpha&=-K^{3/2+\epsilon}, &\beta&=K^{3/2+\epsilon},\\
X&=K^{3.5+\epsilon}, &U&=K^{1.5-\epsilon},\\
Y&=K^{2+\epsilon}, &K^{2-\epsilon}&\ll Q\ll K^{2+\epsilon}.
\end{align*}
Here we mean that there exists a $Q$ in this interval that works. Then the integral is bounded by
$$\ll\frac{QX}{\sqrt Y}\ll K^{5.5-1+\epsilon}=K^{4.5+\epsilon}.$$

We sum after that over $T$ with $K^{2-\epsilon}\ll\det(T)\ll K^{2+\epsilon}$, which gives a contribution of size $K^{2+\epsilon}\cdot K^{-3+\epsilon}\ll K^{-1+2\epsilon}$. The remaining sums are the sum over the other coefficients of $Q$, the one over $d\bmod xc$ and the various $\pm,c,s,U,V$-sums for the exponential sum $H^\pm$. They are all of size $K^\epsilon$. The rank one term is therefore bounded by
$$\ll K^{-4}\cdot K^{-1+\epsilon}\cdot K^\epsilon\cdot K^{4.5+\epsilon}\ll K^{-1/2+3\epsilon}.$$

\section{Rank 2 term}\label{section6}
In this chapter, we focus on the last error term. It comes from  the combination of Equation (\ref{defNav}), (\ref{equNormPreKitaoka}) and the rank 2 term of the Kitaoka formula (Theorem \ref{thmKitaokaFormula}). Its shape is
\begin{align*}
\frac{96\pi^4}{\omega K^4}\sum_{k\in2\mathbb N}&w\left(\frac kK\right)\sum_{T,Q\in\mathcal P(\mathbb Z)/\PSL_2(\mathbb Z)}\frac1{\epsilon(T)\epsilon(Q)\det(TQ)^{3/4}}\\
 &\cdot\int_{\Lambda_{\ev}}V(\det(T),\det(Q),t_\phi,k)\phi(z_T)\bar\phi(z_Q)d\phi\\
 &\cdot\sum_{\det(C)\neq0}\frac{K(Q,T;C)}{|\det(C)|^{3/2}}\mathcal J_\ell(TC^{-1}QC^{-t}).
\end{align*}

Using the estimate $J_k(x)\ll\left(\frac xk\right)^k$ in Equation (\ref{boundsJBessel}), we get
$$J_\ell(TC^{-1}QC^{-t})=\int_0^{\pi/2}J_\ell(4\pi s_1\sin(\theta))J_\ell(4\pi s_2\sin(\theta))\sin(\theta)d\theta\ll\left(\frac{s_1s_2}{k^2}\right)^k.$$

Therefore $k^{2-\epsilon}\ll s_1s_2=\det(TC^{-1}QC^{-t})^{1/2}=\frac{\det(TQ)^{1/2}}{\det(C)}\ll\frac{k^{2+\epsilon}}{\det(C)}$. The last estimate comes from Equation (\ref{eqcutoff}), up to a negligible error. Hence $\det(C)\ll k^\epsilon$ and $k^{4-\epsilon}\ll\det(T)\det(Q)\ll k^{4+\epsilon}$. Using Equation (\ref{eqcutoff}), we also have $\det(T)=\det(Q)(1+O(k^{-1/2+\epsilon}))$.

The restriction on $C$ is a bit subtle because there exist infinitely many matrices with a fixed determinant. We prove later that actually $\Vert C\Vert_\infty\ll k^\epsilon$. Lemma 2 in \cite{Blomer2016spectral} gives us already a bound
\begin{align}\label{equFirstBoundNormC}
\Vert C\Vert^2\ll\Vert T\Vert\Vert Q\Vert.
\end{align}
This is because $s_1\gg k^{1-\epsilon}$, again using Equation (\ref{boundsJBessel}). Since, without loss of generality, $T$ and $Q$ are reduced, we have $\Vert C\Vert^2\ll\det(T)\det(Q)\ll k^{4+\epsilon}$. Recall also that the generalized Kloosterman sum $K(Q,T,C)$ is normalized by the factor $\det(C)^{3/2}$. The goal of the section is to prove that the $C$-sum is short and to detect further cancellation in the $T$ and $Q$ sums coming from the generalized Bessel function $\mathcal J_\ell$. The idea is that if $s_1$ and $s_2$ are far from each other, $\mathcal J_\ell$ should be small. This is made more precise in Subsection \ref{secStationaryPhase}.

\subsection{Summing over $k$}
First we use Lemma \ref{lemSumK} to take advantage of the average over $k$. Let $s_1\geq s_2>0$ the square root of the two eigenvalues of $TC^{-1}QC^{-t}$. We want to analyze the sum
\begin{align*}
\sum_{k\in2\mathbb N}\tilde w(k)&\mathcal J_\ell(TC^{-1}QC^{-t})=\sum_{k\in2\mathbb N}\tilde w(k)\int_0^{\pi/2}J_\ell(4\pi s_1\sin(\alpha))J_\ell(4\pi s_2\sin(\alpha))\sin(\alpha)d\alpha,
\end{align*}
where
$$\tilde w(k)=\tilde w(k,\det(T),\det(Q),t_\phi)=w\left(\frac kK\right)V(\det(T),\det(Q),t_\phi,k)$$
and we temporarily drop the other dependencies. Applying Equation (\ref{equProductOfBesselFunctions}) gives
\begin{align*}
\sum_{k\in2\mathbb N}\tilde w(k)&\mathcal J_\ell(TC^{-1}QC^{-t})=\Re\left(\frac1\pi e\left(-\frac{k-1/2}4\right)\right.\\
	\cdot&\left.\int_0^{\pi/2}\int_0^\infty e\left((s_1^2+s_2^2)t+\frac{\sin(\alpha)^2}t\right)\sum_{k\in2\mathbb N}\tilde w(k)J_\ell(4\pi s_1s_2t)\frac{dt}t\sin(\alpha)d\alpha\right).
\end{align*}

We apply Lemma \ref{lemSumK} to the sum over $k$. Using $w_0(x)\ll_A\min\{k^{-A},x^{-1/2}\}$, we see that the term with $w_0$ is negligible. For the two other terms, we get
$$\Re\left(\frac{e(1/8)}\pi\int_0^{\pi/2}\int_0^\infty e\left((s_1^2+s_2^2)t+\frac{\sin(\alpha)^2}t\pm2s_1s_2t\right) w_\pm(4\pi s_1s_2t)\frac{dt}t\sin(\alpha)d\alpha\right).$$

We forget about the real part and just bound what is inside. We show first a trivial bound for this integral. We use the bounds on $w_\pm$ of Lemma \ref{lemSumK} and the last Equation of (\ref{boundsJBessel}). As stated in the remark after Lemma 20 of \cite{Blomer2019symplectic}, this is also valid for $w_\pm$. We get
\begin{align}\label{boundTrivialGenBessel}\nonumber
I:=\int_0^{\pi/2}\int_0^\infty&e\left((s_1^2+s_2^2)t+\frac{\sin(\alpha)^2}t\pm2s_1s_2t\right)w_\pm(4\pi s_1s_2t)\frac{dt}t\sin(\alpha)d\alpha\\\nonumber
 \ll&\int_0^\infty|w_\pm(4\pi s_1s_2t)|\frac{dt}t\\\nonumber
 \ll&K^2\left(\int_0^{K^\epsilon}\left(1+\frac{K^2}{s_1s_2t}\right)^{-1}\frac{dt}t+\int_{K^\epsilon}^\infty\frac{dt}{t^{3/2}}\right)\\
 \ll&K^{2+\epsilon}.
\end{align}

\subsection{Analysis of the integral and distance between eigenvalues}\label{secStationaryPhase}
\begin{lemma}
Let $w_\pm$ as above, $a>0$, $0<b\ll1$ and $K^{2-\epsilon}\ll c\ll K^{2+\epsilon}$. If $a\gg K^{3\epsilon}$, then
$$\int_0^\infty e\left(at+\frac bt\right)w_\pm(ct)\frac{dt}t\ll_A K^{-A}.$$
\end{lemma}

\begin{proof}
Lemma \ref{lemSumK} says that $\frac{d^j}{dt^j}w_\pm(ct)\ll_{A,j} t^{-j}K^2(1+K^{-\epsilon}/t)^{-A}$ for all $A>0$. By induction, we have that
$$\frac{d^j}{dt^j}\frac{w_\pm(ct)}t\ll_{A,j}t^{-(j+1)}K^2\left(1+\frac1{K^\epsilon t}\right)^{-A}.$$

This is because each derivative add either a derivative on $w_\pm(ct)$ or a $\frac1t$ factor. We integrate by parts multiple time. More precisely, we apply Lemma \ref{lemBKY8.1}. Following the notations there, we have
\begin{align*}
h(t)&=2\pi\left(at+\frac bt\right), &h'(t)&=2\pi\left(a-\frac b{t^2}\right),\\
h^{(j)}(t)&=(-1)^jj!\frac{2\pi b}{t^{j+1}}\quad\text{for }j\geq2,\\
w(t)&=\frac{w_\pm(ct)}t, &w^{(j)}(t)&\ll_{A,j}t^{-(j+1)}K^2\left(1+\frac1{K^\epsilon t}\right)^{-A}.
\end{align*}

The only stationary point is $t_0$ such that $0=h'(t_0)=a-\frac b{t_0^2}$, that is $t_0=\sqrt{\frac ba}$ (it only exists if $a\neq0$). Let suppose that $a\gg K^{3\epsilon}$. Then in particular $t_0\ll K^{-3\epsilon/2}$ since $b\ll1$. But in that part of the $t$-integral, the function $w_\pm$ is negligible. For $t\leq2t_0$, we use the bound on $w_\pm$:
$$\int_0^{2t_0}e\left(at+\frac bt\right)w_\pm(ct)\frac{dt}t\ll_AK^2\int_0^{2t_0}\left(1+\frac1{K^\epsilon t}\right)^{-A}\frac{dt}t\ll_A K^2(K^\epsilon t_0)^A\ll_A K^{2-\epsilon A/2}.$$

For $t\geq 2t_0$, we apply Lemma \ref{lemBKY8.1}. We split everything into dyadic intervals $[\alpha,2\alpha]$ with $\alpha\geq2t_0$. We use the constants
\begin{align*}
X&=\frac{K^2}\alpha, &U&=\alpha,\\
Y&=1\gg\frac b\alpha, &Q&=\alpha,\\
R&=\pi a\leq2\pi\left(a-\frac b{\alpha^2}\right).
\end{align*}

Lemma \ref{lemBKY8.1} gives us
\begin{align*}
\int_\alpha^{2\alpha}e\left(at+\frac bt\right)w_\pm(ct)\frac{dt}t\ll_AK^2(\pi a\alpha)^{-A}\ll_A K^2\cdot K^{-3\epsilon A}\alpha^{-A}.
\end{align*}

This can be summed for a dyadic decomposition of $[2t_0,\infty[$ to get
$$\int_{2t_0}^\infty e\left(at+\frac bt\right)w_\pm(ct)\frac{dt}t\ll_A K^2\cdot K^{-3\epsilon A}\sum_{j=\lfloor\log_2(2t_0)\rfloor}^\infty 2^{-jA}\ll_A K^2\cdot K^{-(3\epsilon/2-3\epsilon)A}.$$

Combining both estimates, we have
$$\int_0^\infty e\left(at+\frac bt\right)w_\pm(ct)\frac{dt}t\ll_A K^{2-\epsilon A/2}.$$
\end{proof}

\begin{remark}
In our case, we have $a=(s_1\pm s_2)^2$, $b=\sin(\alpha)^2$ and $c=4\pi s_1s_2$. Using the various estimates coming from Equations (\ref{eqcutoff}) and (\ref{equFirstBoundNormC}), we see that, up to a negligible error, $a\ll K^{3\epsilon}$. Note that $(s_1+s_2)^2\geq4s_1s_2\gg K^{2-\epsilon}$. So the term with this sign is always negligible.
\end{remark}

\subsection{Size of the $T$, $Q$ and $C$ sums}
We are left to analyze the case $a=(s_1-s_2)^2\ll K^\epsilon$. We changed the value of $\epsilon$ here. In this section and the next ones, we may change again the value of $\epsilon$ from one display to the other. We only do this if the new $\epsilon$ is only a constant multiple of the old one.

The goal of this section is to see which $T$, $Q$ and $C$ satisfy the bound $(s_1-s_2)\ll K^\epsilon$. Note first that if $\lambda_1\geq\lambda_2$ are the two eigenvalues of $M=TC^{-1}QC^{-t}$, then
$$\lambda_1-\lambda_2=s_1^2-s_2^2=(s_1-s_2)(s_1+s_2)\ll K^{1+2\epsilon}.$$

This comes from the fact that $K^{2-\epsilon}\ll s_1s_2\ll K^{2+\epsilon}$, so that $K^{1-\epsilon}\ll s_1+s_2\asymp s_1\asymp s_2\ll K^{1+\epsilon}$. We fix some notations for this section:
\begin{align*}
T&=\begin{pmatrix}a&b\\b&c\end{pmatrix},&
Q&=\begin{pmatrix}x&y\\y&z\end{pmatrix},&
C^{-1}&=(c_{ij}),\\
\tilde Q=C^{-1}QC^{-t}&=\begin{pmatrix}\tilde x&\tilde y\\ \tilde y&\tilde z\end{pmatrix},&
M&=T\tilde Q=TC^{-1}QC^{-t},&
M&=(m_{ij}).
\end{align*}

Note that all numbers are integers or half-integers except for $c_{ij}\in \frac1{\det(C)}\mathbb Z$ and for $\tilde x, \tilde y, \tilde z\in\frac1{2\det(C)^2}\mathbb Z$. But since $|\det(C)|\ll K^\epsilon$, this only creates a negligible difference in terms of estimates for distances between coordinates. So we treat them as if they were integers in the rest of the argument and point out where the difference occurs. Recall also that $T$ and $Q$ are reduced, so $2|b|\leq a\leq c$ and $2|y|\leq x\leq z$. In particular, $K^{2-\epsilon}\ll ac\asymp\det(T)\ll K^{2+\epsilon}$ and similarly for $Q$. Consider $(\lambda_1-\lambda_1)^2$ the square of the difference between the two eigenvalues of $M$. By the quadratic formula, this corresponds to the discriminant of the characteristic polynomial of $M$. We have
$$K^{2+\epsilon}\gg(\lambda_1-\lambda_2)^2=\tr(M)^2-4\det(M)=(m_{11}-m_{22})^2+4m_{12}m_{21}.$$

Inserting the values of the product $T\tilde Q$, we get
$$\Delta=(a\tilde x-c\tilde z)^2+4(a\tilde y+b\tilde z)(b\tilde x+c\tilde y).$$

We rearrange the second term. Completing the squares with respect to $\tilde y$, we have
\begin{align*}
4(a\tilde y+b\tilde z)(b\tilde x+c\tilde y)&=4ac\tilde y^2+4b\tilde y(a\tilde x+c\tilde z)+4b^2\tilde x\tilde z\\
 &=\left(2\sqrt{ac}\tilde y+b\frac{a\tilde x+c\tilde z}{\sqrt{ac}}\right)^2-b^2\frac{(a\tilde x+c\tilde z)^2}{ac}+4b^2\tilde x\tilde z\\
 &=\frac1{ac}(2ac\tilde y+b(a\tilde x+c\tilde z))^2-\frac{b^2}{ac}(a\tilde x-c\tilde z)^2.
\end{align*}

We can do a similar computation by completing the square on $b$. We get
\begin{align}
 \Delta&=(a\tilde x-c\tilde z)^2\left(1-\frac{b^2}{ac}\right)+\frac1{ac}(2ac\tilde y+b(a\tilde x+c\tilde z))^2,\label{equsDelta1}\\
 \Delta&=(a\tilde x-c\tilde z)^2\left(1-\frac{\tilde y^2}{\tilde x\tilde z}\right)+\frac1{\tilde x\tilde z}(2\tilde x\tilde zb+\tilde y(a\tilde x+c\tilde z))^2.\label{equsDelta2}
\end{align}

Since $T$ is reduced, $\frac{b^2}{ac}\leq\frac14$. We also have $\tilde x\tilde z>\tilde y^2$ because $\det(\tilde Q)>0$. So all the squares in Equations (\ref{equsDelta1}) and (\ref{equsDelta2}) must be bounded by $K^{2+\epsilon}$. For the first square, we get
\begin{align}\label{equaxcz}
 a\tilde x-c\tilde z\ll K^{1+\epsilon}.
\end{align}

In particular $K^{2-\epsilon}\ll a\tilde x\sim c\tilde z\ll K^{2+\epsilon}$, i.e. $a\tilde x$ and $c\tilde z$ are of the same size. This is because the product of the terms is of size $\det(T\tilde Q)\gg K^{4-\epsilon}$ (by Equation (\ref{eqcutoff}) and considerations at the beginning of this section). Using that $a\leq c$ and $K^{2-\epsilon}\ll ac\ll K^{2+\epsilon}$, this equation also gives
$$\tilde z=\frac ac\tilde x+O\left(\frac{K^{1+\epsilon}}c\right)\leq\tilde x+O(K^{2\epsilon}).$$

We introduce the notation: $\tilde z\lesssim\tilde x\Leftrightarrow\tilde z\leq\tilde x+O(K^\epsilon)$ as $K\to\infty$. Equation (\ref{equaxcz}) allows us to rearrange the right square:
$$K^{2+\epsilon}\gg2ac\tilde y+b(a\tilde x+c\tilde z)\sim2c(a\tilde y+b\tilde z)\sim2a(c\tilde y+b\tilde x).$$

This gives the two other relations
\begin{align}
 a\tilde y+b\tilde z&\ll\frac{K^{2+\epsilon}}c,\label{equaybz}\\
 c\tilde y+b\tilde x&\ll\frac{K^{2+\epsilon}}a.\label{equcybz}
\end{align}

In particular, we have $\tilde y=-\tilde z\frac ba+O(K^\epsilon)$. Using the relation $2|b|\leq a$, we get $2|\tilde y|\lesssim\tilde z\lesssim\tilde x$. So $\tilde Q$ is almost in a "reversed" reduced form and in particular $K^{2-\epsilon}\ll\tilde x\tilde z\asymp\det(\tilde Q)\asymp\frac{xz}{\det(C^2)}\ll K^{2+\epsilon}$.

\begin{lemma}\label{lemNormC}
Let $\epsilon>0$ and $K\in2\mathbb N$. Let $T,Q\in\mathcal P(\mathbb Z)$ such that $K^{4-\epsilon}\ll\det(TQ)\ll K^{4+\epsilon}$ and $\det(T)-\det(Q)\ll K^{3/2+\epsilon}$, and $C\in M_2(\mathbb Z)$ such that $0\neq\det(C)\ll K^\epsilon$ and $\Vert C\Vert\ll K^{2+\epsilon}$. If
$$\Vert C\Vert\gg K^{2\epsilon}$$
then the integral I in Equation (\ref{boundTrivialGenBessel}) satisfies
$$I\ll_A K^{-A},$$
i.e. $\Vert C\Vert\ll K^{2\epsilon}$ up to a negligible error.
\end{lemma}

\begin{proof}
Following the hypothesis, we see that $C^{-1}$ has coefficients in $\frac1{\det(C)}\mathbb Z$ and $|\det(C^{-1})|\ll1$. Therefore $\det(C)\Vert C^{-1}\Vert=\Vert C\Vert$ for the $\infty$-norm on $M_2(\mathbb R)$. So it is equivalent to prove that $\Vert C^{-1}\Vert\ll K^\epsilon$. We can use the results of this subsection and the last.

The proof relies on the numbers of non-zero entries in $C^{-1}$. Because $\det(C)\neq0$, there are at most two zeros and in that last case, $C^{-1}$ is diagonal or anti-diagonal. Since $\det(C^{-1})\ll1$, the result is obvious in this case. Computing the product $\tilde Q=C^{-1}QC^{-t}$, we have
\begin{align*}
\tilde x&=xc_{11}^2+2yc_{11}c_{12}+zc_{12}^2\\
\tilde z&=xc_{21}^2+2yc_{21}c_{22}+zc_{22}^2.
\end{align*}

The matrix $Q$ is reduced, therefore we have $2|y|c_{11}c_{12}\leq |y|(c_{11}^2+c_{12}^2)\leq\frac12xc_{11}^2+\frac12zc_{12}^2$ and the same for the second equation. We get
\begin{align*}
 \tilde x&\asymp xc_{11}^2+zc_{12}^2,\\
 \tilde z&\asymp xc_{21}^2+zc_{22}^2.
\end{align*}

If $c_{12}c_{22}\neq0$, then we have $\tilde x,\tilde z\gg z\geq x$. Therefore $xz\leq z^2\ll \tilde x\tilde z\asymp\frac{xz}{\det(C)^2}\ll xz$. We deduce that $x\asymp z$ and we must have $\tilde x\asymp\Vert C^{-1}\Vert^2x$ or $\tilde z\asymp\Vert C^{-1}\Vert^2x$. Then $\frac{xz}{\det(C)^2}\asymp\tilde x\tilde z\gg\Vert C^{-1}\Vert^2xz$ and so $\Vert C^{-1}\Vert^2\ll\det(C)^{-2}\ll1$.

If $c_{12}=0$, then we have $\tilde x\gtrsim\tilde z\gg z\geq x$. So $\frac{xz}{\det(C)^2}\asymp\tilde x\tilde z\gg z(z+O(K^\epsilon))$ and $z\geq x\gg z+O(K^\epsilon)$. Therefore $z\asymp x$ and we can finish as above.
 
The last case is $c_{22}=0$. We have $\frac{xz}{\det(C)^2}\asymp\tilde x\tilde z\gg c_{12}^2c_{21}^2xz$ so $c_{12},c_{21}\ll1$ and $\tilde z\asymp xc_{21}^2\asymp x$. For $\tilde y$, we have
$$\tilde y=c_{11}c_{21}x+(c_{11}c_{22}+c_{12}c_{21})y+c_{12}c_{22}z=c_{21}(c_{11}x+c_{12}y).$$

Let suppose that $c_{11}\gg K^{2\epsilon}$, so that
$$\tilde z+O(K^\epsilon)\gg|\tilde y|\asymp c_{21}c_{11}x\gg K^{2\epsilon}c_{21}x\asymp K^{2\epsilon}\tilde z.$$

This is a contradiction. Therefore $\Vert C\Vert\ll K^{2\epsilon}.$
\end{proof}

The next lemma is a way to decouple the relationship between the variables.

\begin{lemma}
Let $\epsilon>0$ and $K\in2\mathbb N$. Let $T,Q\in\mathcal P(\mathbb Z)$ such that $K^{4-\epsilon}\ll\det(TQ)\ll K^{4+\epsilon}$ and $\det(T)-\det(Q)\ll K^{3/2+\epsilon}$, and $C\in M_2(\mathbb Z)$ such that $\det(C)\neq0$ and $\Vert C\Vert\ll K^{\epsilon}$. If
$$ac-\det(C)^2\tilde x\tilde z\gg K^{3/2+\epsilon}$$
or
$$a-\det(C)\tilde z\gg K^{-1/2+\epsilon}\tilde z$$
then the integral I in Equation (\ref{boundTrivialGenBessel}) satisfies
$$I\ll_A K^{-A},$$
that is, up to a negligible error,
\begin{align}
ac&=\det(C)^2\tilde x\tilde z+O(K^{3/2+\epsilon})\label{equacC2xz},\\
a&=\det(C)\tilde z+O(K^{-1/2+\epsilon}\tilde z)\label{equaz}.
\end{align}
\end{lemma}

\begin{proof}
We know that
\begin{align}
\det(T)&=\det(Q)+O(K^{3/2+\epsilon})\nonumber\\
\Leftrightarrow ac-b^2&=\det(C)^2(\tilde x\tilde z-\tilde y^2)+O(K^{3/2+\epsilon}).\label{equDet}
\end{align}

We want to simplify this using the other equations in this section. We multiply the Equations (\ref{equaybz}) and (\ref{equcybz}) together.
\begin{align*}
(ac)^2\tilde y^2&=(c\tilde zb+O(K^{2+\epsilon}))(a\tilde xb+O(K^{2+\epsilon}))\\
 &=ac\tilde x\tilde zb^2+O(bK^{2+\epsilon}(a\tilde x+c\tilde z)+K^{4+2\epsilon}),\\
\Rightarrow\tilde y^2&=\frac{\tilde x\tilde z}{ac}\tilde b^2+O(bK^{4\epsilon}+K^{2\epsilon}).
\end{align*}

We simplified the big O term using Equation (\ref{equaxcz}) that tells us that $K^{2-\epsilon}\ll a\tilde x\asymp c\tilde z\ll K^{2+\epsilon}$ and $ac\gg K^{2-\epsilon}$. Inserting this result in Equation (\ref{equDet}), we get
\begin{align*}
ac-b^2&=\det(C)^2(\tilde x\tilde z-\tilde y^2)+O(K^{3/2+\epsilon})\\
 &=\det(C)^2\left(\tilde x\tilde z-\frac{\tilde x\tilde z}{ac}b^2+O(bK^\epsilon+K^{\epsilon})\right)+O(K^{3/2+\epsilon})\\
 &=\det(C)^2\frac{\tilde x\tilde z}{ac}(ac-b^2)+O(K^{3/2+\epsilon}).
\end{align*}

We have $0\neq ac-b^2\asymp ac$. Multiplying by $\frac{ac}{ac-b^2}$, we get
\begin{align*}
ac&=\det(C)^2\tilde x\tilde z+O(K^{3/2+\epsilon}).\
\end{align*}

By subtraction of Equation (\ref{equDet}), we also get $b^2=\det(C)^2y^2+O(K^{3/2+\epsilon})$. Combining Equations (\ref{equaxcz}) and (\ref{equacC2xz}) and recalling that $ac,c\tilde z\gg K^{2-\epsilon}$ (Equations (\ref{eqcutoff}) and (\ref{equaxcz})), we have
\begin{align*}
\det(C)^2\tilde z(a\tilde x-c\tilde z)&\ll K^{1+\epsilon}\det(C)^2\tilde z,\\
\makebox[0pt][l]{\underline{\phantom{$+\quad a(ac-\det(C)^2\tilde x\tilde z)\ll K^{3/2}a a loooooooonger line$}}} 
+\quad a(ac-\det(C)^2\tilde x\tilde z)&\ll K^{3/2+\epsilon}a,\\
ca^2-\det(C)^2c\tilde z^2&\ll K^{1+\epsilon}\det(C)^2\tilde z+K^{3/2+\epsilon}a,\\
a^2-\det(C)^2\tilde z^2&\ll K^{1+3\epsilon}\frac{\tilde z}c+K^{3/2+\epsilon}\frac ac\ll K^{-1+4\epsilon}\tilde z^2+K^{-1/2+2\epsilon}a^2,\\
a-\det(C)\tilde z&\ll K^{-1/2+2\epsilon}a\asymp K^{-1/2+3\epsilon}\tilde z.
\end{align*}

The last equation comes from the observation that, on the line above, the two bounds on the right are smaller than a term on the left. Therefore the two terms must be of the same size. We can then factorize the left-hand side and simplify.
\end{proof}

\begin{remark}
Similarly, we can prove that $c-\det(C)\tilde x\ll K^{-1/2+\epsilon}\tilde x$ up to a negligible error.
\end{remark}

\subsection{Estimate of the rank 2 term}
First, we analyze the sum coming from the Fourier series and the spectral integral. Each term in the $T,Q$-sum has the following shape. We suppose for the following argument that $C$ is fixed.
\begin{align*}
\frac1{\epsilon(T)\epsilon(Q)\det(TQ)^{3/4}}&\int_0^{\pi/2}\int_0^\infty e\left((s_1-s_2)^2t+\frac{\sin(\alpha)^2}t\right)\int_{\Lambda_{\ev}}w_-(\pi s_1s_2t)\phi(z_T)\bar\phi(z_Q)d\phi\,\frac{dt}t\sin(\alpha)d\alpha\\
	&\ll\det(TQ)^{-3/4}\cdot K^\epsilon\cdot K^{2}\cdot K^{1/2+\epsilon}\frac{\det(TQ)^{1/4}}{\sqrt{ax}}\\
	&\ll K^{1/2+3\epsilon}\frac1{\sqrt{ax}}
\end{align*}

We used the following estimates. Recall that $\det(T),\det(Q)\gg K^{2-\epsilon}$ and $\epsilon(T),\epsilon(Q)\ll1$ up to a negligible error. The integrals over $t$ and $\alpha$ are of size $K^\epsilon$, as seen in Equation (\ref{boundTrivialGenBessel}) (with $K^2$ being the size of $w_\pm$). The spectral integral is bounded using Lemma \ref{lemBoundNonDiagSpect}.

Now, we count the number of $T$ and $Q$ using the cut-off we computed. First, we fix $Q$. This also fix $\tilde Q$ since $C$ is fixed. We fix $a$, $c$ and $b$ in this order. Equations (\ref{equaz}), (\ref{equaxcz}) and (\ref{equaybz}) give respectively
\begin{align*}
a&=\det(C)\tilde z+O\left(\frac{\tilde z}{K^{1/2-\epsilon}}\right),\\
c&=\frac{a\tilde x}{\tilde z}+O\left(\frac{K^{1+\epsilon}}{\tilde z}\right),\\
b&=\frac{a\tilde y}{\tilde z}+O\left(\frac{K^{2+\epsilon}}{c\tilde z}\right).
\end{align*}

Note that in the big O of the first equation, the fraction can be smaller than 1. We also know that $c\tilde z\gg K^{2-\epsilon}$. Therefore to fix $T$, we have
$$O\left(\left(\frac{\tilde z}{K^{1/2-\epsilon}}+1\right)\cdot \frac{K^{1+\epsilon}}{\tilde z}\cdot \frac{K^{2+\epsilon}}{c\tilde z}\right)=O\left(\left(K^{-1/2+\epsilon}+\frac1{\tilde z}	\right)K^{1+3\epsilon}\right)$$
possible choices. We use the divisor bound, Equation (\ref{equaz}), $\tilde z\ll K^{1+\epsilon}$ and $a,x\gg1$ to get that the $T,Q$-sum is of size
\begin{align*}
\sum_{K^{2-\epsilon}\ll\tilde x\tilde z\ll K^{2+\epsilon}}&\sum_{2|\tilde y|\ll\tilde z+O(K^\epsilon)}\left(K^{-1/2+\epsilon}+\frac1{\tilde z}\right)K^{1+\epsilon}\cdot K^{1/2+\epsilon}\\
	&\ll\sum_{K^{2-\epsilon}\ll\tilde x\tilde z\ll K^{2+\epsilon}}\left(\frac{\tilde z}{K^{1/2-\epsilon}}+1\right)K^{3/2+3\epsilon}\\
	&\ll K^{2+\epsilon}\cdot K^{1/2+2\epsilon}\cdot K^{1+4\epsilon}\\
	&\ll K^{3.5+7\epsilon}.
\end{align*}

We have estimated the $T,Q$-sum, as always up to a negligible error. Now we combine this with other estimates to bound the term of rank 2. Note that there are $O(K^\epsilon)$ choices for $C$ by Lemma \ref{lemNormC}. We get 
\begin{align*}
\frac{96\pi^3}{\omega K^4}&\sum_{T,Q\in\mathcal P(\mathbb Z)/\PSL_2(\mathbb Z)}\frac1{\epsilon(T)\epsilon(Q)\det(TQ)^{3/4}}\sum_{\substack{\Vert C\Vert\ll1\\\det(C)\neq0}}{|\det(C)|^{3/2}}\int_0^{\pi/2}\int_0^\infty e\left((s_1-s_2)^2t+\frac{\sin(\alpha)^2}t\right)\\
	&\quad\cdot \int_{\Lambda_{\ev}}w_-(\pi s_1s_2t)\phi(z_T)\bar\phi(z_Q)d\phi\,\frac{dt}t\sin(\alpha)d\alpha\\
	&\ll K^{-4}\cdot K^\epsilon\cdot K^{3.5+\epsilon}\\
	&\ll K^{-1/2+2\epsilon}.
\end{align*}

This proves the bound on the rank two term. Together with the results of Sections \ref{section4} and \ref{section5}, it concludes the proof of Theorem \ref{thmNav}.

\appendix
\section{Automorphism of binary quadratic forms}\label{AppAutomorphisms}
The goal of this appendix is to compute all the automorphisms in $\GL_2(\mathbb Z)$ of a binary quadratic form. We set
\begin{align*}
Q&=\begin{pmatrix}x&y\\y&z\end{pmatrix}
&M&=\begin{pmatrix}a&b\\c&d\end{pmatrix}.
\end{align*}

Here $Q$ is a (weakly) reduced integral quadratic form, that is $x,z\neq0$, $2|y|\leq x\leq z$, $2y,x,z\in\mathbb Z$ and $\det(Q)>0$, and $M\in\GL_2(\mathbb Z)$. We are looking for the couples $(Q,M)$ such that
$$Q=M^tQM.$$

Note first that if we replace $M$ by $-M$, we get the same result. Therefore we only consider matrices up to multiplication by $\pm1$. The computation gives
\begin{align}\label{AppEquMatrixAutomorphisms}
0=M^tQM-Q=\begin{pmatrix}
a^2x+2acy+c^2z-x &abx+(ad+bc)y+cdz-y\\
abx+(ad+bc)y+cdz-y &b^2x+2bdy+d^2z-z
\end{pmatrix}.
\end{align}

We consider the first entry. Using the identity $u^2+v^2\geq2|uv|$, we have
$$0=a^2x+2acy+c^2z-x\geq2|ac|(\sqrt{xz}-|y|)-x\geq|ac|x-x.$$

Therefore we have $|ac|\leq1$. We have to work a bit more for the last entry. Suppose that $|d|\geq2$. Then $d^2-1\geq\frac34d^2$ and so
$$0=b^2x+2bdy+(d^2-1)z\geq2|b|\sqrt{d^2-1}\sqrt{xz}-2|bdy|\geq2|bd|(\sqrt{3/4}\sqrt{xz}-|y|).$$

Since $\sqrt{3/4}>1/2$, we have $b=0$. Therefore we have two cases: $|d|\leq1$ or $b=0$.

\subsection{Diagonal and antidiagonal $M$}
We begin with the two easy cases of diagonal and antidiagonal matrix $M$. There are 4 possibilities up to multiplication by $-1$:
$$M=\begin{pmatrix}1&0\\0&1\end{pmatrix},\begin{pmatrix}1&0\\0&-1\end{pmatrix},\begin{pmatrix}0&1\\-1&0\end{pmatrix},\begin{pmatrix}0&1\\1&0\end{pmatrix}.$$

The identity is an automorphism for any matrix $Q$. Looking at Equation (\ref{AppEquMatrixAutomorphisms}), we get respectively for the other three matrices
$$0=\begin{pmatrix}0&-2y\\-2y&0\end{pmatrix},\begin{pmatrix}z-x&-2y\\-2y&x-z\end{pmatrix},\begin{pmatrix}x-z&0\\0&z-x\end{pmatrix}.$$

Therefore the conditions on $Q$ are respectively $y=0$, $x=z\wedge y=0$ and $x=z$.

\subsection{Diagonal $Q$}
We quickly consider the case $y=0$, so we can rule out this later. Equation (\ref{AppEquMatrixAutomorphisms}) rewrites as
$$0=\begin{pmatrix}a^2x+c^2z-x &abx+cdz\\abx+cdz &b^2x+d^2z-z\end{pmatrix}.$$

First, if $a=0$, then $b,c=\pm1$ since the determinant is $bc=\pm1$. The first entry gives $x=z$ and the second entry gives $d=0$. If $c=0$, then $a,d=\pm1$ and the diagonal entries vanish. The second entry gives $b=0$. In both cases, we are back to a diagonal or antidiagonal $M$. Otherwise, if $ac=\pm1$, then the first entry gives $z=0$ which is a contradiction. So all these cases fit in the last section. From now, we suppose that $y\neq0$

\subsection{The case $ac=0$}
If $c=0$, then automatically $a$ and $d$ equal $\pm1$ since the determinant is $ad$. That gives the matrices
$$M=\begin{pmatrix}1&n\\0&1\end{pmatrix},\begin{pmatrix}1&n\\0&-1\end{pmatrix}$$

for $n$ a non-zero integer. The other cases can be obtained by multiplying by $-1$. Looking at Equation (\ref{AppEquMatrixAutomorphisms}), we have
$$0=\begin{pmatrix}
0 &anx+(ad-1)y\\
anx+(ad-1)y &n^2x+2dny
\end{pmatrix}.$$

So if $ad=1$ like in the first case, then $x=0$ and there is no such $Q$. In the second case, $ad=-1$ and we get $nx=2y$ or $nx+2y=0$. Since $x\geq2|y|$, we get $n=\sgn(y)$ and $x=2|y|$. Now, if $a=0$ then $bc=\pm1$ and we have the matrices
$$M=\begin{pmatrix}0&1\\1&n\end{pmatrix},\begin{pmatrix}0&1\\-1&n\end{pmatrix}.$$

Equation (\ref{AppEquMatrixAutomorphisms}) rewrite as
$$0=\begin{pmatrix}z-x&(bc-1)y+cnz\\(bc-1)y+cnz&x+2bny+(n^2-1)z\end{pmatrix}.$$

If $bc=1$, then $z=0$ and there is no such matrix. Otherwise, $x=z$ and we get the two equations $nx=2y$ and $nx+2y=0$. Again, $x\geq2|y|$ so $n=-\sgn(y)$ and $x=2|y|$.

\subsection{The case $ac=1$}
We have $a=c=\pm1$, without loss of generality say $a=c=1$. Therefore the first entry of the matrix is $2y+z=0$. Since $2|y|\leq x\leq z$, we get $-2y=x=z$. Equation (\ref{AppEquMatrixAutomorphisms}) rewrites as
$$0=\begin{pmatrix}0 &-by-dy-y\\-by-dy-y &-2b^2y+2bdy-2(d^2-1)y\end{pmatrix}.$$

If $b=0$, then the second entry gives $d+1=0$ so $d=-1$ and this is compatible with the last entry. If $b\neq0$, then we have two cases. If $d=0$, then the second equation gives $b=-1$. This is compatible with the last entry. If $d=\pm1$, then the last entry is $-2b^2y+2bdy=0$, so that $b=d$. There is no such matrix with determinant $\pm1$ and it is also incompatible with the second entry.

\subsection{The case $ac=-1$}
We have $a=-c=\pm1$, without loss of generality say $a=-c=1$. So the first entry of Equation (\ref{AppEquMatrixAutomorphisms}) gives $2y=z$. Since $2|y|\leq x\leq z$, we have $2y=x=z$. The full matrix rewrites
$$0=\begin{pmatrix}0 &by-dy-y\\by-dy-y &2b^2y+2bdy+2d^2y-2y\end{pmatrix}.$$

If $b=0$, then the second entry gives $d=-1$ and is compatible with the last. If $b\neq0$, then $d=0$ gives $b=1$ for both equations. If $d=\pm1$, then the last entry is $2b^2y+2bdy=0$ so $b=-d$. This is incompatible with the second entry that says $b=d+1$ (for integral $b$ and $d$).

\subsection{Summary}
We summarize the result in the table below. The first column indicates the sign of the determinant of $M$. For each matrix $M$, there is the matrix $-M$ that has the same action on $Q$. Note that except for the fourth entry, $y$ is always supposed to be non-zero.

$$
\begin{array}{|c|c|c|}
\hline
\det(M) &M &Q\\
\hline
\hline
+&\begin{pmatrix}1&0\\0&1\end{pmatrix} &\text{Any}\\
\hline
-&\begin{pmatrix}1&0\\0&-1\end{pmatrix} &\begin{pmatrix}x&0\\0&z\end{pmatrix}\\
\hline
+&\begin{pmatrix}0&1\\-1&0\end{pmatrix} &\begin{pmatrix}x&0\\0&x\end{pmatrix}\\
\hline
-&\begin{pmatrix}0&1\\1&0\end{pmatrix} &\begin{pmatrix}x&y\\y&x\end{pmatrix}\\
\hline

\hline
-&\begin{pmatrix}1&1\\0&-1\end{pmatrix} &\begin{pmatrix}2y&y\\y&z\end{pmatrix}\\
\hline
-&\begin{pmatrix}1&-1\\0&-1\end{pmatrix} &\begin{pmatrix}2y&-y\\-y&z\end{pmatrix}\\
\hline
+&\begin{pmatrix}0&1\\-1&1\end{pmatrix} &\begin{pmatrix}2y&y\\y&2y\end{pmatrix}\\
\hline
+&\begin{pmatrix}0&1\\-1&-1\end{pmatrix} &\begin{pmatrix}2y&-y\\-y&2y\end{pmatrix}\\
\hline

\hline
-&\begin{pmatrix}1&0\\1&-1\end{pmatrix} &\begin{pmatrix}2y&-y\\-y&2y\end{pmatrix}\\
\hline
+&\begin{pmatrix}1&-1\\1&0\end{pmatrix} &\begin{pmatrix}2y&-y\\-y&2y\end{pmatrix}\\
\hline

\hline
-&\begin{pmatrix}1&0\\-1&-1\end{pmatrix} &\begin{pmatrix}2y&y\\y&2y\end{pmatrix}\\
\hline
+&\begin{pmatrix}1&1\\-1&0\end{pmatrix} &\begin{pmatrix}2y&y\\y&2y\end{pmatrix}\\
\hline
\end{array}
$$

We rewrite this table in terms of $Q$. The second column lists all the automorphisms of $Q$ (modulo $\pm id$). The three following columns indicates respectively the number of automorphisms in $\SL_2(\mathbb Z)$, in $\GL_2(\mathbb Z)$ and the ratio between the two. The number of automorphisms $\epsilon(Q)$ in $\PSL_2(\mathbb Z)$ is just half of the number in $\SL_2(\mathbb Z)$. The last column gives the corresponding Heegner point $z=\frac{-y+i\sqrt{xz-y^2}}{x}$. Here $y\neq0$ everywhere and $y>0$ except in the third row. Recall that if $Q$ is reduced and $x=z$ or $x=2|y|$, we can, furthermore, suppose that $y>0$. This removes the fifth and the seventh rows.

$$
\hspace{-1cm}
\begin{array}{|c|c|c|c|c|c|}
\hline
Q &M &\SL_2(\mathbb Z) &\GL_2(\mathbb Z) &\text{Ratio} &\text{Heegner pt}\\
\hline
\hline
\begin{pmatrix}x&0\\0&z\end{pmatrix}&
\begin{pmatrix}1&0\\0&1\end{pmatrix},
\begin{pmatrix}1&0\\0&-1\end{pmatrix}
&2 &4 &2 &i\sqrt{\frac zx}\\
\hline
\begin{pmatrix}x&0\\0&x\end{pmatrix}&
\begin{pmatrix}1&0\\0&1\end{pmatrix},
\begin{pmatrix}1&0\\0&-1\end{pmatrix},
\begin{pmatrix}0&1\\-1&0\end{pmatrix},
\begin{pmatrix}0&1\\1&0\end{pmatrix}
&4 &8 &2 &i\\
\hline
\begin{pmatrix}x&y\\y&x\end{pmatrix}&
\begin{pmatrix}1&0\\0&1\end{pmatrix},
\begin{pmatrix}0&1\\1&0\end{pmatrix}
&2 &4 &2 &\frac{-y+i\sqrt{x^2-y^2}}x\\
\hline
\begin{pmatrix}2y&y\\y&z\end{pmatrix}&
\begin{pmatrix}1&0\\0&1\end{pmatrix},
\begin{pmatrix}1&1\\0&1\end{pmatrix}
&2 &4 &2 &\frac{-1}2+i\frac{\sqrt{2z-y}}{2\sqrt y}\\
\hline
\begin{pmatrix}2y&-y\\-y&z\end{pmatrix}&
\begin{pmatrix}1&0\\0&1\end{pmatrix},
\begin{pmatrix}1&-1\\0&-1\end{pmatrix}
&2 &4 &2 &\frac{1}2+i\frac{\sqrt{2z-y}}{2\sqrt y}\\
\hline
\begin{pmatrix}2y&y\\y&2y\end{pmatrix}&
\begin{pmatrix}1&0\\0&1\end{pmatrix},
\begin{pmatrix}0&1\\1&0\end{pmatrix},
\begin{pmatrix}1&1\\0&-1\end{pmatrix},
\begin{pmatrix}0&1\\-1&1\end{pmatrix},
\begin{pmatrix}1&0\\-1&-1\end{pmatrix},
\begin{pmatrix}1&1\\-1&0\end{pmatrix}
&6 &12 &2 &\frac{-1+i\sqrt3}2\\
\hline
\begin{pmatrix}2y&-y\\-y&2y\end{pmatrix}&
\begin{pmatrix}1&0\\0&1\end{pmatrix},
\begin{pmatrix}0&1\\1&0\end{pmatrix},
\begin{pmatrix}1&-1\\0&-1\end{pmatrix},
\begin{pmatrix}0&1\\-1&-1\end{pmatrix},
\begin{pmatrix}1&0\\1&-1\end{pmatrix},
\begin{pmatrix}1&-1\\1&0\end{pmatrix}
&6 &12 &2 &\frac{1+i\sqrt3}2\\
\hline
\text{Other}&
\begin{pmatrix}1&0\\0&1\end{pmatrix}
&2 &2 &1 &\frac{-y+i\sqrt{xz-y^2}}x\\
\hline
\end{array}
$$

\end{document}